\documentclass[links,onefignum,onetabnum]{siamart220329}

\usepackage{amsfonts,amssymb}

\newsiamremark{assumption}{Assumption}
\crefname{assumption}{Assumption}{Assumptions}

\headers{Zero-Sum Games for Volterra Integral Equations}{M. I. Gomoyunov}

\title{Zero-Sum Games for Volterra Integral Equations and Viscosity Solutions of Path-Dependent Hamilton--Jacobi Equations\thanks{Submitted to the editors April 16, 2024.
\funding{This work was supported by the Russian Science Foundation, project~21-71-10070, \\ \href{https://rscf.ru/en/project/21-71-10070/}{https://rscf.ru/en/project/21-71-10070/}}}}

\author{Mikhail I. Gomoyunov\thanks{N. N. Krasovskii Institute of Mathematics and Mechanics of the Ural Branch of the Russian Academy of Sciences, Ekaterinburg, 620108, Russia, and Ural Federal University, Ekaterinburg, 620002, Russia (\email{m.i.gomoyunov@gmail.com}).}}

\ifpdf
\hypersetup{
  pdftitle={Zero-Sum Games for Volterra Integral Equations},
  pdfauthor={M. I. Gomoyunov}
}
\fi

 \newcommand{\rd}{\mathrm{d}}
 \newcommand{\dist}{\operatorname{dist}}
 \newcommand{\Id}{\mathrm{Id}_n}
 \newcommand{\diam}{\operatorname{diam}}
 \newcommand{\argmin}[1]{\underset{#1}{\operatorname{argmin} \,}}
 \newcommand{\argmax}[1]{\underset{#1}{\operatorname{argmax} \,}}

\begin{document}

\maketitle

\begin{abstract}
     We consider a game, in which the dynamics is described by a non-linear Volterra integral equation of Hammerstein type with a weakly-singular kernel and the goals of the first and second players are, respectively, to minimize and maximize a given cost functional.
     We propose a way of how the dynamic programming principle can be formalized and the theory of generalized (viscosity) solutions of  path-dependent Hamilton--Jacobi equations can be developed in order to prove the existence of the game value, obtain a characterization of the value functional, and construct players' optimal feedback strategies.
\end{abstract}



\section{Introduction}

    In the theory of zero-sum differential games (see, e.g., \cite{Isaacs_1965,Friedman_1971,Krasovskii_Subbotin_1988,Bardi_Capuzzo-Dolcetta_1997,Fleming_Soner_2006,Yong_2015}), the following two problems are among the most important ones.
    The first problem is to prove the existence of the game value, that is, the coincidence of the lower and upper game values, which are usually defined through players' non-anticipative strategies (also called Elliott--Kalton strategies \cite{Elliott_Kalton_1972}, see also, e.g., \cite{Fleming_Soner_2006,Yong_2015}).
    The second problem is to design methods for constructing players' optimal feedback controls (also called Krasovskii--Subbotin or positional control strategies \cite{Krasovskii_Subbotin_1988,Subbotin_1995}), which allow the players to ensure the achievement of the game value.
    To studying and solving both problems, several approaches were proposed.
    Among others, the well-established one is based on the dynamic programming principle and uses the results from the theory of Hamilton--Jacobi equations and their generalized solutions (for example, in the minimax \cite{Subbotin_1995} or viscosity \cite{Crandall_Lions_1983,Crandall_Evans_Lions_1984} sense).
    In the present paper, we develop this approach in order to investigate the above two problems for a broader class of zero-sum games in which the dynamic equation is a Volterra integral equation.

    Compared to differential games, dynamic games for systems described by Volterra integral equations have been studied quite little in the literature.
    Here, let us note that, in \cite{Pasikov_1986,Pasikov_2015}, a game with linear dynamics is considered and a method for constructing players' optimal positional strategies is proposed in a so-called regular case.
    In \cite{You_1999}, a linear-quadratic game (called ``integral game'' by analogy with the term ``differential game'') is investigated.
    In \cite{Chernov_2012,Chernov_2014}, a game is considered in which each player controls its own system described by a rather general Volterra functional operator equation and the problem of existence of the game value and $\varepsilon$-equilibria in the classes of so-called piecewise open-loop strategies of the players is studied by applying the technique going back to, e.g., \cite[Section 2.3]{Petrosjan_1993}.
    In \cite{Carlson_2017}, the existence of an open-loop Nash equilibria in nonzero-sum games is established.
    Thus, to the best of our knowledge, there has been no research devoted directly to the development of the dynamic programming principle and the theory of Hamilton--Jacobi equations for zero-sum games with the dynamics described by Volterra integral equations.
    This paper aims to contribute to this field.

    In addition, it should be noted that optimal control problems for Volterra integral equations (which can be treated as zero-sum games with a fictitious second player) are quite widely considered in the literature.
    The main theoretical questions are related to necessary optimality conditions for open-loop controls.
    The reader is referred to relatively recent papers \cite{Bonnans_de_la_Vega_Dupuis_2013,Dmitruk_Osmolovski_2017}, where more references to earlier works on this topic can also be found.
    Let us separately dwell on papers \cite{Belbas_2007,Belbas_2008} that touched upon questions related to the dynamic programming principle and the Hamilton--Jacobi equations in such optimal control problems.
    More precisely, in \cite{Belbas_2007}, the difficulties that arise in the case under consideration are indicated and it is proposed, in particular, to take pairs consisting of the current time and the control history up to this time as a so-called ``position'' of the system.
    Let us emphasize that, being applied to a zero-sum game, this approach leads to the fact that the current time and the histories of control actions of both players should be taken as a position of the system.
    Therefore, this does not seem reasonable enough since one player's knowledge of the entire opponent's control history may not be available, especially if the opponent is interpreted as an unknown disturbance, which is often the case in applications.
    In paper \cite{Belbas_2008}, which is a continuation of \cite{Belbas_2007}, the approach based on the reduction of the dynamic equation to an infinite-dimensional system of ordinary differential equations and subsequent application of the dynamic programming principle is developed.
    However, this approach requires certain smoothness properties of the integrand and continuity of controls.
    Furthermore, let us point out that, recently, there has been increased interest in necessary optimality conditions in optimal control problems for weakly-singular Volterra integral equations, where the integrand may have a singularity of power-law type (see, e.g., \cite{Lin_Yong_2020,Han_Lin_Yong_2023,Idczak_2023,Moon_2022,Moon_2023,Gasimov_Asadzade_Mahmudov_2023}).
    This is particularly related to investigations of optimal control problems for systems described by differential equations with fractional-order derivatives.
    So, the results of the present paper also contribute to the field of optimal control problems for both regular and weakly-singular Volterra integral equations.

    More precisely, in this paper, we deal with a finite-horizon game, in which the dynamics is described by a second-kind non-linear Volterra integral equation of Hammerstein type with a kernel $K$, which may have a weak singularity of power-law type, and the goals of the first and second players are, respectively, to minimize and maximize a given cost functional of rather general form.
    For the theory and applications of such integral equations, the reader is referred to, e.g., \cite{Tricomi_1957,Corduneanu_1991,Gorenflo_Vessella_1991,Brunner_2017}.
    Let us note that the considered class of integral equations include, as particular cases, the integral equations that correspond to ordinary differential equations as well as to multi-order Caputo fractional differential equations (see, e.g., \cite{Bourdin_2018}).

    The first objective of the paper is to prove that the game has the value, i.e., to prove that the lower and upper game values defined in the classes of non-anticipative strategies of the players coincide.
    To this end, we follow a rather standard approach mentioned above.
    Nevertheless, each of its steps requires a modification, as is shortly described below.

    At the first step, we introduce a notion of a position of the system as a pair $(t, w(\cdot))$ consisting of a time $t \in [0, T]$ and a function $w \colon [0, t] \to \mathbb{R}^n$ (from a certain function space), treated as a history of a system motion on the time interval $[0, t]$.
    Here, $T > 0$ is the fixed time horizon and $n \in \mathbb{N}$ is the dimension of the system state.
    Since systems described by Volterra integral equations are of hereditary nature (in other words, have a memory effect), such a view on the system positions seems quite natural.
    In this regard, an analogy can be traced with differential games for time-delay (see, e.g., \cite{Osipov_1971,Lukoyanov_2000_PMM_Eng,Kaise_2015}) and fractional-order (see, e.g., \cite{Gomoyunov_2021_Mathematics}) systems.

    At the second step, for every position $(t, w(\cdot))$ chosen as an initial one, we provide a statement of the game and define the lower and upper values of this game.
    Thus, we obtain the lower and upper value functionals defined on the space $\mathcal{G}$ of all system positions.
    The principal point is that these functionals satisfy the equations of the dynamic programming principle.
    Recall that, in a general case, a number of difficulties are associated with the application of the dynamic programming principle if the system is described by a Volterra integral equation.
    Nevertheless, as the present paper shows, it turns out that this can be done under some additional assumptions.
    In this connection, the key assumption made in the paper is that the linear Volterra integral operator with the same kernel $K$ is non-degenerate, i.e., has trivial kernel (some easily verified sufficient conditions for the fulfillment of this assumption are also presented).
    In particular, this assumption is used from the very beginning in order to define the space $\mathcal{G}$ of system positions $(t, w(\cdot))$ and provide the statement of the game for every position $(t, w(\cdot)) \in \mathcal{G}$ in a suitable way.
    Then, after establishing a semigroup property of system motions, the proof of the dynamic programming principle for the lower and upper value functionals can be carried out by rather standard arguments (see, e.g., \cite[Theorem 3.1]{Evans_Souganidis_1984} and also \cite[Chapter XI, Theorem 5.1]{Fleming_Soner_2006}, \cite[Theorem 3.3.5]{Yong_2015}).

    At the third step, for an arbitrary functional defined on the space $\mathcal{G}$, we introduce a notion of coinvariant differentiability ($ci$-differentiability for short).
    This notion depends on the kernel $K$ and generalizes the usual notion of $ci$-differentiability \cite{Kim_1999,Lukoyanov_2000_PMM_Eng} (see also the discussion in \cite[Section 5.2]{Gomoyunov_Lukoyanov_Plaksin_2021}) as well as the notion of fractional $ci$-diffe\-rentiability \cite{Gomoyunov_2020_SIAM}.
    The latter two notions proved to be a useful tool in developing the theory of path-dependent Hamilton--Jacobi equations associated to differential games for time-delay (see, e.g., \cite{Lukoyanov_2000_PMM_Eng,Lukoyanov_2004_PMM_Eng,Kaise_2015,Gomoyunov_Lukoyanov_Plaksin_2021,Plaksin_2021_SIAM,Kaise_2022}) and fractional-order (see, e.g., \cite{Gomoyunov_2021_Mathematics,Gomoyunov_Lukoyanov_2021,Gomoyunov_2023_JDE}) systems.
    In addition, we introduce a notion of a $ci$-smooth functional and obtain a formula for the total derivative of a $ci$-smooth functional along system motions.

    At the fourth step, we consider an abstract path-dependent Hamilton--Jacobi equation with the $ci$-derivatives and formulate a Cauchy problem for this equation and a right-end boundary condition.
    Following \cite{Crandall_Lions_1983,Crandall_Evans_Lions_1984} and, in the path-dependent setting, \cite{Soner_1988,Lukoyanov_2007_IMM_Eng,Gomoyunov_2023_JDE}, we give a definition of a viscosity solution of the Cauchy problem through $ci$-smooth test functionals and a sequence of compact subsets $\mathcal{G}_k$, $k \in \mathbb{N}$, of the space $\mathcal{G}$, each of which is invariant with respect to system motions and the union of which covers the whole space $\mathcal{G}$.
    Furthermore, we require that the restriction of a viscosity solution to each set $\mathcal{G}_k$ should be continuous and should satisfy a certain Lipschitz continuity condition in the functional variable $w(\cdot)$.
    Then, after verifying these continuity properties for the lower and upper value functionals, we derive from the dynamic programming principle and the established formula for the total derivative of $ci$-smooth functionals along system motions that the lower value functional (respectively, the upper value functional) is a viscosity solution of the Cauchy problem for the Hamilton--Jacobi equation with the lower Hamiltonian (respectively, the upper Hamiltonian) and a natural boundary condition in a rather standard way (see, e.g., \cite[Theorem 4.1]{Evans_Souganidis_1984} and also \cite[Chapter XI, Theorem 6.1]{Fleming_Soner_2006}, \cite[Theorem 3.3.6]{Yong_2015}).

    At the final step, we prove a theorem on uniqueness of viscosity solutions of the Cauchy problems under consideration.
    The proof almost completely repeats that of \cite[Theorem 5.1]{Gomoyunov_2023_JDE} (see also \cite{Lukoyanov_2007_IMM_Eng,Plaksin_2021_SIAM}), which is a similar result for Cauchy problems for path-dependent Hamilton--Jacobi equations with fractional $ci$-derivatives.
    The only thing that needs to be done is to establish that the Lyapunov--Krasovskii functional $\nu_\varepsilon$ from \cite{Gomoyunov_2023_JDE}, which is used to construct the required $ci$-smooth test functionals, has appropriate properties.
    As a corollary of this theorem, we find that, under the Isaacs condition (i.e., when the lower and upper Hamiltonians coincide), the game has the value.
    Furthermore, we actually obtain a stronger statement that the value functional is characterized as a unique viscosity solution of the corresponding Cauchy problem.
    This result is a generalization of \cite[Corollary 5.1]{Gomoyunov_2023_JDE} to the more general class of dynamical systems studied in the present paper.

    Non-anticipative strategies, being a convenient tool in theoretical considerations, are rather difficult to implement, especially if one of the players is interpreted as an unknown disturbance.
    In this regard, the second objective of the paper is to propose a way of constructing players' optimal positional strategies, which are more acceptable from a practical point of view.
    To this end, we follow \cite{Garnysheva_Subbotin_1994_Eng} and \cite[Section 12.2]{Subbotin_1995} (see also \cite{Lukoyanov_2004_PMM_Eng} and \cite{Gomoyunov_2021_Mathematics,Gomoyunov_2023_Motor} for the cases for of time-delay and fractional-order systems respectively) and perform some ``smoothing'' transformation of the value functional of the game.
    This allows us to determine certain extremal directions that we then use in the extremal aiming procedure \cite{Krasovskii_Subbotin_1988}.
    The basis for the transformation is again the Lyapunov--Krasovskii functional $\nu_\varepsilon$ from \cite{Gomoyunov_2023_JDE}.

    The paper is organized as follows.
    \Cref{section_statement} presents the statement of the game.
    The additional assumption is formulated and discussed in \cref{section_assumption}.
    In \cref{section_set_of_positions}, the space of system positions $\mathcal{G}$ is defined and the extended problem statement is given.
    In \cref{section_DPP}, the semigroup property of system motions and the dynamical programming equations for the lower and upper value functionals are described.
    The auxiliary sets $\mathcal{G}_k$, $k \in \mathbb{N}$, are defined and their properties are established in \cref{section_sequence}.
    In \cref{section_continuity}, the continuity properties of the lower and upper value functionals are obtained.
    In \cref{section_ci_derivatives}, the new notion of $ci$-differentiability is introduced and the formula for the total derivative of a $ci$-smooth functional is presented.
    In \cref{section_HJ}, the Cauchy problem for the path-dependent Hamilton--Jacobi equation is considered, the definition of a viscosity solution of this problem is given, and the theorem on uniqueness of the viscosity solutions is proved.
    Moreover, it is shown that the lower and upper value functionals are viscosity solutions of the corresponding Cauchy problems, which yields the existence of the game value and the characterization of the value functional.
    In \cref{section_positional}, the classes of players' positional strategies are described.
    In \cref{section_optimal_positional}, the method for constructing the corresponding optimal strategies is proposed and justified.
    Some concluding remarks are given in \cref{section_conclusion}.

\section{Problem Statement}
\label{section_statement}

    Let $n \in \mathbb{N}$ and $T > 0$.
    Let $\mathbb{R}^n$ be the Euclidean space of $n$-dimensional vectors with the inner product $\langle \cdot, \cdot \rangle$ and the norm $\|\cdot\|$, and let $\mathbb{R}^{n \times n}$ be the space of ($n \times n$)-dimensional matrices endowed with the corresponding induced norm, also denoted by $\|\cdot\|$.
    Given $R > 0$, let $B(R)$ be the closed ball in $\mathbb{R}^n$ centered at the origin of radius $R$.
    Let $C([0, T], \mathbb{R}^n)$ be the classical space of continuous functions from $[0, T]$ to $\mathbb{R}^n$ endowed with the norm $\|\cdot\|_{C([0, T], \mathbb{R}^n)}$.
    Denote
    \begin{displaymath}
        \Omega
        \doteq \bigl\{ (\tau, \xi) \in [0, T] \times [0, T] \colon \tau \geq \xi \bigr\},
        \quad \Omega^\circ
        \doteq \bigl\{ (\tau, \xi) \in \Omega \colon \tau > \xi \bigr\}.
    \end{displaymath}

    Consider a {\it dynamical system} described by the Volterra integral equation
    \begin{equation} \label{system}
        x(\tau)
        = y(\tau) + \int_{0}^{\tau} K(\tau, \xi) f(\xi, x(\xi), u(\xi), v(\xi)) \, \rd \xi.
    \end{equation}
    Here, $\tau \in [0, T]$ is time, $T$ is the time horizon;
    $x(\tau) \in \mathbb{R}^n$ is the state of the system at time $\tau$;
    $y(\cdot) \in C([0, T], \mathbb{R}^n)$ is a fixed function;
    $u(\xi) \in P$ and $v(\xi) \in Q$ are, respectively, the controls of the first and second players at time $\xi$,
    $P \subset \mathbb{R}^{n_P}$ and $Q \subset \mathbb{R}^{n_Q}$ are compact sets, $n_P$, $n_Q \in \mathbb{N}$.

    \begin{assumption} \label{assumption_f}
        The function $f \colon [0, T] \times \mathbb{R}^n \times P \times Q \to \mathbb{R}^n$ satisfies the conditions below.

        \noindent (a)
            The function $f$ is continuous.

        \noindent (b)
            For any $R > 0$, there exists $\lambda > 0$ such that
            \begin{displaymath}
                \|f(\tau, x, u, v) - f(\tau, x^\prime, u, v)\|
                \leq \lambda \|x - x^\prime\|
            \end{displaymath}
            for all $\tau \in [0, T]$, $x$, $x^\prime \in B(R)$, $u \in P$, and $v \in Q$.

        \noindent (c)
            There exists $c > 0$ such that
            \begin{displaymath}
                \|f(\tau, x, u, v)\|
                \leq c (1 + \|x\|)
                \quad \forall \tau \in [0, T], \ x \in \mathbb{R}^n, \ u \in P, \ v \in Q.
            \end{displaymath}
    \end{assumption}

    \begin{assumption} \label{assumption_K}
        The kernel $K \colon \Omega^\circ \to \mathbb{R}^{n \times n}$ satisfies the conditions below.

        \noindent (a)
            The function $K$ can be represented in the form
            \begin{equation} \label{K_representation}
                K(\tau, \xi)
                = \frac{K_\ast(\tau, \xi)}{(\tau - \xi)^{1 - \alpha}}
                \quad \forall (\tau, \xi) \in \Omega^\circ
            \end{equation}
            with a continuous function $K_\ast \colon \Omega \to \mathbb{R}^{n \times n}$ and a number $\alpha \in (0, 1)$.

        \noindent (b)
            For the function $K_\ast$ from \cref{K_representation}, there exist $\beta \in (0, 1]$ and $\lambda > 0$ such that
            \begin{displaymath}
                \|K_\ast(\tau, \xi) - K_\ast(\tau, \xi^\prime)\|
                \leq \lambda |\xi - \xi^\prime|^\beta
                \quad \forall (\tau, \xi), (\tau, \xi^\prime) \in \Omega.
            \end{displaymath}
    \end{assumption}

    Note that the representation of $K$ in form \cref{K_representation} is not unique.
    In particular, we assume that $\alpha$ is less than $1$, since if there is such a representation with $\alpha = 1$, we can always take arbitrarily $\alpha^\prime \in (0, 1)$, put $K_\ast^\prime(\tau, \xi) \doteq (\tau - \xi)^{1 - \alpha^\prime} K_\ast(\tau, \xi)$, $(\tau, \xi) \in \Omega$, and get the another representation
    \begin{displaymath}
        K(\tau, \xi)
        = \frac{K_\ast^\prime(\tau, \xi)}{(\tau - \xi)^{1 - \alpha^\prime}}
        \quad \forall (\tau, \xi) \in \Omega^\circ
    \end{displaymath}
    with the function $K_\ast^\prime$ satisfying the required continuity assumptions.

    An additional assumption on $K$ is made in \cref{section_assumption} below (see \cref{assumption_K_2}).

    Let us denote by $\mathcal{U}[0, T]$ and $\mathcal{V}[0, T]$ the sets of (Lebesgue) measurable functions $u \colon [0, T] \to P$ and $v \colon [0, T] \to Q$ respectively.
    Any functions $u(\cdot) \in \mathcal{U}[0, T]$ and $v(\cdot) \in \mathcal{V}[0, T]$ are considered as admissible {\it open-loop controls} on the time interval $[0, T]$ of the first and second players.
    A {\it motion} of system \cref{system} generated by a pair of controls $u(\cdot) \in \mathcal{U}[0, T]$ and $v(\cdot) \in \mathcal{V}[0, T]$ is defined as a function $x(\cdot) \in C([0, T], \mathbb{R}^n)$ that, together with $u(\cdot)$ and $v(\cdot)$, satisfies the integral equation \cref{system} for all $\tau \in [0, T]$.

    \begin{proposition} \label{proposition_existence_uniqueness}
        For any players' controls $u(\cdot) \in \mathcal{U}[0, T]$ and $v(\cdot) \in \mathcal{V}[0, T]$, there exists a unique motion $x(\cdot) \doteq x(\cdot; u(\cdot), v(\cdot))$ of system \cref{system}.
    \end{proposition}

    Despite the fact that \cref{proposition_existence_uniqueness} can be proved by a rather standard arguments, we present the proof below for the reader's convenience.
    Before doing this, let us recall the following result (see, e.g., \cite[Theorem 4.3.2]{Gorenflo_Vessella_1991}), which is valid due to \cref{assumption_K}, (a), and used several times in the paper:
    for every $p \in (1 / \alpha, \infty]$, the {\it linear Volterra integral operator} $\mathbf{K}_p \colon L_p([0, T], \mathbb{R}^n) \to C([0, T], \mathbb{R}^n)$ given by
    \begin{equation} \label{operator_K}
        \mathbf{K}_p [\ell(\cdot)](\tau)
        \doteq \int_{0}^{\tau} K(\tau, \xi) \ell(\xi) \, \rd \xi
        \quad \forall \tau \in [0, T], \ \ell(\cdot) \in L_p([0, T], \mathbb{R}^n)
    \end{equation}
    is well-defined and compact (i.e., $\mathbf{K}_p$ is continuous and maps bounded subsets of $L_p([0, T], \mathbb{R}^n)$ into relatively compact subsets of $C([0, T], \mathbb{R}^n)$).
    Here, $L_p([0, T], \mathbb{R}^n)$ is the classical Lebesgue space of (classes of equivalence) of measurable functions from $[0, T]$ to $\mathbb{R}^n$ endowed with the norm $\|\cdot\|_{L_p([0, T], \mathbb{R}^n)}$.

    \begin{proof}[Proof of \cref{proposition_existence_uniqueness}]
        Fix $u(\cdot) \in \mathcal{U}[0, T]$, $v(\cdot) \in \mathcal{V}[0, T]$ and define the operator $\mathbf{F} \colon C([0, T], \mathbb{R}^n) \to C([0, T], \mathbb{R}^n)$ by
        \begin{displaymath}
            \mathbf{F}[x(\cdot)](\tau)
            \doteq y(\tau) + \int_{0}^{\tau} K(\tau, \xi) f(\xi, x(\xi), u(\xi), v(\xi)) \, \rd \xi
        \end{displaymath}
        for all $\tau \in [0, T]$ and $x(\cdot) \in C([0, T], \mathbb{R}^n)$.
        For every $x(\cdot) \in C([0, T], \mathbb{R}^n)$, the function $\ell(\xi) \doteq f(\xi, x(\xi), u(\xi), v(\xi))$, $\xi \in [0, T]$, is measurable and bounded by \cref{assumption_f}, (a).
        Therefore, taking into account that $\mathbf{F}[x(\cdot)](\tau) = y(\tau) + \mathbf{K}_\infty[\ell(\cdot)](\tau)$, $\tau \in [0, T]$, and $y(\cdot) \in C([0, T], \mathbb{R}^n)$, we find that $\mathbf{F}$ is well-defined.
        Thus, in order to complete the proof, it suffices to show that $\mathbf{F}$ has a unique fixed point.

        The proof of the existence of a fixed point relies on the Leray--Shauder fixed-point theorem (see, e.g., \cite[Theorem 6.A]{Zeidler_1986}).
        Let us verify that $\mathbf{F}$ is compact.
        Suppose that $\{x^{(i)}(\cdot)\}_{i = 1}^\infty \subset C([0, T], \mathbb{R}^n)$, $x^\ast(\cdot) \in C([0, T], \mathbb{R}^n)$, and $\|x^{(i)}(\cdot) - x^\ast(\cdot)\|_{C([0, T], \mathbb{R}^n)} \to 0$ as $i \to \infty$.
        Take $R > 0$ such that $\|x^\ast(\cdot)\|_{C([0, T], \mathbb{R}^n)} \leq R$ and $\|x^{(i)}(\cdot)\|_{C([0, T], \mathbb{R}^n)} \leq R$, $i \in \mathbb{N}$.
        Choose the corresponding number $\lambda$ by \cref{assumption_f}, (b).
        Then, denoting
        \begin{equation} \label{mu_ast}
            \kappa_\ast
            \doteq \max_{(\tau, \xi) \in \Omega} \|K_\ast(\tau, \xi)\|,
        \end{equation}
        we derive
        \begin{equation} \label{proposition_existence_uniqueness_p_continuity}
            \begin{aligned}
                & \|\mathbf{F}[x^{(i)}(\cdot)](\tau) - \mathbf{F}[x^\ast(\cdot)](\tau)\| \\
                & \quad \leq \int_{0}^{\tau} \|K(\tau, \xi)\| \| f(\xi, x^{(i)}(\xi), u(\xi), v(\xi)) - f(\xi, x^\ast(\xi), u(\xi), v(\xi)) \| \, \rd \xi \\
                & \quad \leq \kappa_\ast \lambda \int_{0}^{\tau} \frac{\|x^{(i)}(\xi) - x^\ast(\xi)\|}{(\tau - \xi)^{1 - \alpha}} \, \rd \xi
                \leq \frac{\kappa_\ast \lambda T^\alpha}{\alpha} \|x^{(i)}(\cdot) - x^\ast(\cdot)\|_{C([0, T], \mathbb{R}^n)}
            \end{aligned}
        \end{equation}
        for all $\tau \in [0, T]$ and, hence, $\|\mathbf{F}[x^{(i)}(\cdot)](\cdot) - \mathbf{F}[x^\ast(\cdot)](\cdot)\|_{C([0, T], \mathbb{R}^n)} \to 0$ as $i \to \infty$.
        So, $\mathbf{F}$ is continuous.

        Further, let $M > 0$ and $S \doteq \{x(\cdot) \in C([0, T], \mathbb{R}^n) \colon \|x(\cdot)\|_{C([0, T], \mathbb{R}^n)} \leq M\}$.
        Owing to \cref{assumption_f}, (a), there exists $M^\prime > 0$ such that $\|f(\xi, x(\xi), u(\xi), v(\xi))\| \leq M^\prime$ for all $\xi \in [0, T]$ and $x(\cdot) \in S$.
        Therefore, we obtain the inclusion $\mathbf{F}[S] \subset \{y(\cdot)\} + \mathbf{K}_\infty[E]$, where $E \doteq \{\ell(\cdot) \in L_\infty([0, T], \mathbb{R}^n) \colon \|\ell(\cdot)\|_{L_\infty([0, T], \mathbb{R}^n)} \leq M^\prime\}$.
        Consequently, the set $\mathbf{F}[S]$ is relatively compact in view of compactness of the operator $\mathbf{K}_\infty$.

        Finally, let us show that the required a priory estimate is fulfilled with the number $N \doteq (1 + \|y(\cdot)\|_{C([0, T], \mathbb{R}^n)}) \mathrm{E}_\alpha(\mathrm{\Gamma}(\alpha) \kappa_\ast c T^\alpha) - 1$, where $\kappa_\ast$ is given by \cref{mu_ast}, $c$ is taken from \cref{assumption_f}, (c), $\mathrm{E}_\alpha$ is the Mittag-Leffler function, and $\mathrm{\Gamma}$ is the Euler gamma function.
        Let $x(\cdot) \in C([0, T], \mathbb{R}^n)$ and $\gamma \in (0, 1)$ be such that $x(\cdot) = \gamma \mathbf{F}[x(\cdot)](\cdot)$.
        Then,
        \begin{displaymath}
            \|x(\tau)\|
            \leq \|y(\cdot)\|_{C([0, T], \mathbb{R}^n)}
            + \kappa_\ast c \int_{0}^{\tau} \frac{1 + \|x(\xi)\|}{(\tau - \xi)^{1 - \alpha}} \, \rd \xi
            \quad \forall \tau \in [0, T],
        \end{displaymath}
        wherefrom, by the generalized Gronwall inequality (see, e.g., \cite[Lemma 1.3.13]{Brunner_2017}), we get $\|x(\cdot)\|_{C([0, T], \mathbb{R}^n)} \leq N$.
        As a result, we conclude that $\mathbf{F}$ has a fixed point.

        To prove the uniqueness part, suppose that $x(\cdot)$, $x^\prime(\cdot) \in C([0, T], \mathbb{R}^n)$ are two fixed points of $\mathbf{F}$.
        Consider $R > 0$ such that $\|x(\cdot)\|_{C([0, T], \mathbb{R}^n)} \leq R$, $\|x^\prime(\cdot)\|_{C([0, T], \mathbb{R}^n)} \leq R$ and choose the corresponding number $\lambda$ according to \cref{assumption_f}, (b).
        Arguing similarly to \cref{proposition_existence_uniqueness_p_continuity}, we derive
        \begin{displaymath}
                \|x(\tau) - x^\prime(\tau)\|
                = \|\mathbf{F}[x(\cdot)](\tau) - \mathbf{F}[x^\prime(\cdot)](\tau)\|
                \leq \kappa_\ast \lambda \int_{0}^{\tau} \frac{\|x(\xi) - x^\prime(\xi)\|}{(\tau - \xi)^{1 - \alpha}} \, \rd \xi
        \end{displaymath}
        for all $\tau \in [0, T]$.
        Hence, applying \cite[Lemma 1.3.13]{Brunner_2017} again, we obtain $x(\cdot) = x^\prime(\cdot)$.
    \end{proof}

    Let us point out two {\it particular cases} of the integral equation \cref{system}.
    Suppose that
    \begin{equation} \label{y_constant}
        y(\tau) = y_0
        \quad \forall \tau \in [0, T]
    \end{equation}
    for some $y_0 \in \mathbb{R}^n$.
    Then, if
    \begin{equation} \label{ordinary}
        K(\tau, \xi)
        = \Id
        \quad \forall (\tau, \xi) \in \Omega^\circ,
    \end{equation}
    where $\Id \in \mathbb{R}^{n \times n}$ is the identity matrix, then the integral equation \cref{system} corresponds to the Cauchy problem for the ordinary differential equation
    \begin{displaymath}
        \dot{x}(\tau)
        = f(\tau, x(\tau), u(\tau), v(\tau))
        \quad \text{for a.e. } \tau \in [0, T]
    \end{displaymath}
    under the initial condition $x(0) = y_0$, where $\dot{x}(\tau) \doteq \rd x(\tau) / \rd \tau$.
    As an another example, suppose that, for every $(\tau, \xi) \in \Omega^\circ$, the matrix $K(\tau, \xi)$ is diagonal and its diagonal elements are of the form
    \begin{equation} \label{fractional}
        k_{i, i}(\tau, \xi)
        = \frac{1}{\mathrm{\Gamma}(\alpha_i) (\tau - \xi)^{1 - \alpha_i}}
        \quad \forall i \in \overline{1, n}
    \end{equation}
    for some $\alpha_i \in (0, 1]$, $i \in \overline{1, n}$.
    In this case, the integral equation \cref{system} corresponds to the Cauchy problem for the fractional differential equation
    \begin{displaymath}
        (^C D^{\boldsymbol{\alpha}} x)(\tau)
        = f(\tau, x(\tau), u(\tau), v(\tau))
        \quad \text{for a.e. } \tau \in [0, T]
    \end{displaymath}
    under the initial condition $x(0) = y_0$, where $(^C D^{\boldsymbol{\alpha}} x)(\tau)$ is the Caputo fractional derivative of multi-order $\boldsymbol{\alpha} \doteq \{\alpha_i\}_{i =1}^{n}$ (see, e.g., \cite{Bourdin_2018}).

    For the dynamical system \cref{system}, a {\it game} is studied in which the first player tries to minimize while the second player tries to maximize the {\it cost functional}
    \begin{equation} \label{cost_functional}
        J(u(\cdot), v(\cdot))
        \doteq \sigma(x(\cdot))
        + \int_{0}^{T} \chi(\tau, x(\tau), u(\tau), v(\tau)) \, \rd \tau,
    \end{equation}
    where $u(\cdot) \in \mathcal{U}[0, T]$, $v(\cdot) \in \mathcal{V}[0, T]$, and $x(\cdot) \doteq x(\cdot; u(\cdot), v(\cdot))$ is the system motion.

    \begin{assumption} \label{assumption_sigma_chi}
        For $\sigma \colon C([0, T], \mathbb{R}^n) \to \mathbb{R}$ and $\chi \colon [0, T] \times \mathbb{R}^n \times P \times Q \to \mathbb{R}$, the conditions below are satisfied.

        \noindent (a)
            The function $\chi$ is continuous.

        \noindent (b)
            For every $R > 0$, there exists $\lambda > 0$ such that
            \begin{displaymath}
                |\chi(\tau, x, u, v) - \chi(\tau, x^\prime, u, v)|
                \leq \lambda \|x - x^\prime\|
                \quad \forall \tau \in [0, T], \ x, x^\prime \in B(R), \ u \in P, \ v \in Q.
            \end{displaymath}

        \noindent (c)
            For any compact set $\mathcal{D} \subset C([0, T], \mathbb{R}^n)$, there exists $\lambda > 0$ such that
            \begin{displaymath}
                |\sigma(x(\cdot)) - \sigma(x^\prime(\cdot))|
                \leq \lambda \biggl( \|x(T) - x^\prime(T)\| + \int_{0}^{T} \|x(\tau) - x^\prime(\tau)\| \, \rd \tau \biggr)
                \quad \forall x(\cdot), x^\prime(\cdot) \in \mathcal{D}.
            \end{displaymath}
    \end{assumption}

    In addition, it is supposed that the following assumption holds, which is called the {\it Isaacs condition} in the differential games theory.
    \begin{assumption} \label{assumption_Isaacs}
        For any $\tau \in [0, T]$ and $x$, $s \in \mathbb{R}^n$, the equality below is valid:
        \begin{displaymath}
            \min_{u \in P} \max_{v \in Q} h(\tau, x, u, v, s)
            = \max_{v \in Q} \min_{u \in P} h(\tau, x, u, v, s),
        \end{displaymath}
        where we denote
        \begin{equation} \label{h}
            h(\tau, x, u, v, s)
            \doteq \langle s, f(\tau, x, u, v) \rangle + \chi(\tau, x, u, v)
        \end{equation}
        for all $\tau \in [0, T]$, $x$, $s \in \mathbb{R}^n$, $u \in P$, and $v \in Q$.
    \end{assumption}

    Let us define the lower and upper values of the game \cref{system,cost_functional}.
    For the first player, a {\it non-anticipative strategy} is a mapping $\boldsymbol{a} \colon \mathcal{V}[0, T] \to \mathcal{U}[0, T]$ with the following property: for any $t \in [0, T]$ and any controls $v(\cdot)$, $v^\prime(\cdot) \in \mathcal{V}[0, T]$ of the second player, if the equality $v(\tau) = v^\prime(\tau)$ holds for a.e. $\tau \in [0, t]$, then the corresponding controls $u(\cdot) \doteq \boldsymbol{a}[v(\cdot)](\cdot)$ and $u^\prime(\cdot) \doteq \boldsymbol{a}[v^\prime(\cdot)](\cdot)$ of the first player satisfy the equality $u(\tau) = u^\prime(\tau)$ for a.e. $\tau \in [0, t]$.
    Then, the {\it lower value of the game} is given by
    \begin{equation} \label{lower_value}
        \rho_-^0
        \doteq \inf_{\boldsymbol{a} \in \boldsymbol{\mathcal{A}}[0, T]} \sup_{v(\cdot) \in \mathcal{V}[0, T]}
        J \bigl( \boldsymbol{a}[v(\cdot)](\cdot), v(\cdot) \bigr),
    \end{equation}
    where $\boldsymbol{\mathcal{A}}[0, T]$ is the set of first player's non-anticipative strategies $\boldsymbol{a}$.
    Similarly, a {\it second player's non-anticipative strategy} is a mapping $\boldsymbol{b} \colon \mathcal{U}[0, T] \to \mathcal{V}[0, T]$ such that, for any $t \in [0, T]$ and $u(\cdot)$, $u^\prime(\cdot) \in \mathcal{U}[0, T]$, if $u(\tau) = u^\prime(\tau)$ for a.e. $\tau \in [0, t]$, then $v(\tau) = v^\prime(\tau)$ for a.e. $\tau \in [0, t]$, where $v(\cdot) \doteq \boldsymbol{b}[u(\cdot)](\cdot)$ and $v^\prime(\cdot) \doteq \boldsymbol{b}[u^\prime(\cdot)](\cdot)$.
    So, the {\it upper game value} is
    \begin{equation} \label{upper_value}
        \rho_+^0
        \doteq \sup_{\boldsymbol{b} \in \boldsymbol{\mathcal{B}}[0, T]} \inf_{u(\cdot) \in \mathcal{U}[0, T]} J \bigl( u(\cdot), \boldsymbol{b}[u(\cdot)](\cdot) \bigr),
    \end{equation}
    where $\boldsymbol{\mathcal{B}}[0, T]$ is the set of second player's non-anticipative strategies $\boldsymbol{b}$.
    If the lower and upper game values coincide, it is said that the game \cref{system,cost_functional} has the {\it value}
    \begin{equation} \label{game_value}
        \rho^0
        \doteq \rho_-^0
        = \rho_+^0.
    \end{equation}

\section{Additional assumption}
\label{section_assumption}

    The next assumption on the kernel $K$ from \cref{system} is crucial for the validity of the results obtained in the paper.
    \begin{assumption} \label{assumption_K_2}
        For every fixed $t \in (0, T]$, the linear Volterra integral operator $\mathbf{K}_\infty^{[t]} \colon L_\infty([0, t], \mathbb{R}^n) \to C([0, t], \mathbb{R}^n)$ defined by (see also \cref{operator_K})
        \begin{equation} \label{operator_K_t}
            \mathbf{K}_\infty^{[t]} [\ell(\cdot)](\tau)
            \doteq \int_{0}^{\tau} K(\tau, \xi) \ell(\xi) \, \rd \xi
            \quad \forall \tau \in [0, t], \ \ell(\cdot) \in L_\infty([0, t], \mathbb{R}^n)
        \end{equation}
        has trivial kernel, i.e., if $\ell(\cdot) \in L_\infty([0, t], \mathbb{R}^n)$ and $\mathbf{K}_\infty^{[t]}[\ell(\cdot)](\tau) = 0$ for all $\tau \in [0, t]$, then $\ell(\xi) = 0$ for a.e. $\xi \in [0, t]$.
    \end{assumption}

    The proposition below provides two easily verified sufficient conditions under which \cref{assumption_K_2} is fulfilled.
    \begin{proposition} \label{proposition_K}
        Suppose that the kernel $K$ satisfy \cref{assumption_K}, {\rm (a)}.
        Then, \cref{assumption_K_2} is fulfilled if one of the following conditions is met.

        \noindent {\rm (a)}
            For any $(\tau, \xi) \in \Omega^\circ$ and $i$, $j \in \overline{1, n}$ with $i < j$, the equality $k_{i, j}(\tau, \xi) = 0$ is valid, and, for every $i \in \overline{1, n}$, it holds that $k_{i, i}(\tau, \xi) = p_i (\tau - \xi)$ for all $(\tau, \xi) \in \Omega^\circ$ with some function $p_i \colon (0, T] \to \mathbb{R}$ such that $p_i(\tau) \neq 0$ for a.e. $\tau \in (0, T]$.
            Here, $k_{i, j}(\tau, \xi)$ is the corresponding element of the matrix $K(\tau, \xi)$.

        \noindent {\rm (b)}
            For every $\tau \in [0, T]$, the matrix $K_\ast(\tau, \tau)$ from \cref{K_representation} is non-degenerate, and the partial derivative $\partial K_\ast(\tau, \xi) / \partial \tau$ exists for all $(\tau, \xi) \in \Omega$ and is continuous on $\Omega$.
    \end{proposition}
    \begin{proof}
        Let us show that (a) implies \cref{assumption_K_2}.
        Suppose that $t \in (0, T]$, $\ell(\cdot) \in L_\infty([0, t], \mathbb{R}^n)$, and $\mathbf{K}_\infty^{[t]}[\ell(\cdot)](\tau) = 0$, $\tau \in [0, t]$.
        Then, by the first part of (a),
        \begin{equation} \label{first_sufficient_condition_system}
            \sum_{j = 1}^{i} \int_{0}^{\tau} k_{i, j}(\tau, \xi) \ell_j(\xi) \, \rd \xi
            = 0
            \quad \forall \tau \in [0, t], \ i \in \overline{1, n},
        \end{equation}
        where $\ell_j(\cdot)$ denotes the corresponding coordinate function of $\ell(\cdot)$.
        Substituting $i = 1$ into \cref{first_sufficient_condition_system} and using the second part of (a), we get
        \begin{displaymath}
            \int_{0}^{\tau} p_1 (\tau - \xi) \ell_1(\xi) \, \rd \xi
            = 0
            \quad \forall \tau \in [0, t].
        \end{displaymath}
        Therefore, applying the Titchmarsh convolution theorem (see \cite[Theorem VII]{Titchmarsh_1926} and also, e.g., \cite{Doss_1988}), we obtain $\ell_1(\xi) = 0$ for a.e. $\xi \in [0, t]$.
        Taking this into account and substituting $i = 2$ into \cref{first_sufficient_condition_system}, we derive
        \begin{displaymath}
            \int_{0}^{\tau} p_2 (\tau - \xi) \ell_2(\xi) \, \rd \xi
            = 0
            \quad \forall \tau \in [0, t]
        \end{displaymath}
        and similarly deduce that $\ell_2(\xi) = 0$ for a.e. $\xi \in [0, t]$.
        Continuing this process up to $i = n$, we find that $\ell(\xi) = 0$ for a.e. $\xi \in [0, t]$, and, hence, \cref{assumption_K_2} holds.

        The proof of the fact that (b) implies \cref{assumption_K_2} can be carried out according to scheme from, e.g., \cite[Theorem 1.5.7]{Brunner_2017} (see also \cite[Theorem 5.1.3]{Gorenflo_Vessella_1991}).
    \end{proof}

    In particular, \cref{proposition_K} allows us to conclude that the kernels from \cref{ordinary,fractional} satisfy \cref{assumption_K_2}.

    The role of \cref{assumption_K_2} is to ensure the following property of motions of system \cref{system} (for the exact statement, see \cref{proposition_semigroup_property}).
    Let us take $t \in (0, T)$ and suppose that the players control the system in two steps:
    at the initial time $\tau = 0$, they choose their controls $u(\cdot) \in \mathcal{U}[0, t]$ and $v(\cdot) \in \mathcal{V}[0, t]$ on the time interval $[0, t]$ only;
    after that, at the intermediate time $t$, they choose their controls $u^\prime(\cdot) \in \mathcal{U}[t, T]$ and $v^\prime(\cdot) \in \mathcal{V}[t, T]$ on the remaining time interval $[t, T]$.
    Then, \cref{assumption_K_2} implies that the information about the history of the motion $x(\cdot)$ on $[0, t]$ and about the controls $u^\prime(\cdot)$ and $v^\prime(\cdot)$ is sufficient to correctly determine $x(\cdot)$ on $[t, T]$.

    To illustrate that this property may fail to be valid without \cref{assumption_K_2}, let us consider the dynamical system described by the Volterra integral equation
    \begin{displaymath}
        x(\tau)
        = K(\tau) \int_{0}^{\tau} (u(\xi) + v(\xi)) \, \rd \xi,
    \end{displaymath}
    where $\tau \in [0, T]$, $x(\tau) \in \mathbb{R}$, $u(\xi) \in [- 1, 1]$, $v(\xi) \in [- 1, 1]$, and
    \begin{displaymath}
        K(\tau)
        \doteq
        \begin{cases}
            0 & \forall \tau \in [0, t], \\
            \tau - t & \forall \tau \in (t, T].
        \end{cases}
    \end{displaymath}
    It is clear that, for any players' controls $u(\cdot) \in \mathcal{U}[0, t]$ and $v(\cdot) \in \mathcal{V}[0, t]$, we have $x(\tau) = 0$, $\tau \in [0, t]$, i.e., the history of the motion $x(\cdot)$ on $[0, t]$ does not depend on $u(\cdot)$ and $v(\cdot)$.
    On the other hand, choosing $u^\prime(\tau) \doteq 0$, $v^\prime(\tau) \doteq 0$, $\tau \in [t, T]$, we obtain
    \begin{displaymath}
        x(\tau)
        = (\tau - t) \int_{0}^{t} (u(\xi) + v(\xi)) \, \rd \xi
        \quad \forall \tau \in [t, T],
    \end{displaymath}
    i.e., the motion $x(\cdot)$ on $[t, T]$ already depends on the initial choice of the controls $u(\cdot)$ and $v(\cdot)$.
    Thus, in the considered example, the information about the history of the motion $x(\cdot)$ on $[0, t]$ and about the controls $u^\prime(\cdot)$ and $v^\prime(\cdot)$ is not sufficient to correctly determine the motion $x(\cdot)$ on $[t, T]$, and the additional information about the controls $u(\cdot)$ and $v(\cdot)$ is required.

    Note that another comment regarding \cref{assumption_K_2} is given at the end of \cref{section_optimal_positional} below.

\section{System positions and extended problem statement}
\label{section_set_of_positions}

    Given $t \in (0, T]$, let us introduce the set $\mathcal{W}[0, t]$ consisting of functions $w \colon [0, t] \to \mathbb{R}^n$ for each of which there exists a function $\ell(\cdot) \in L_\infty([0, t], \mathbb{R}^n)$ such that
    \begin{equation} \label{representation_x}
        w(\tau)
        = y(\tau) + \int_{0}^{\tau} K(\tau, \xi) \ell(\xi) \, \rd \xi
        \quad \forall \tau \in [0, t],
    \end{equation}
    where $y(\cdot)$ and $K$ are taken from \cref{system}.
    Equivalently, using the notation from \cref{operator_K_t}, we can define $\mathcal{W}[0, t] = \{y_t(\cdot)\} + \mathbf{K}_\infty^{[t]}[L_\infty([0, t], \mathbb{R}^n)]$.
    Here and below, for a function $x \colon [0, T] \to \mathbb{R}^n$, we denote by $x_t(\cdot)$ its restriction to $[0, t]$, i.e.,
    \begin{equation} \label{x_t}
        x_t(\tau)
        \doteq x(\tau)
        \quad \forall \tau \in [0, t].
    \end{equation}
    In particular, $\mathcal{W}[0, t] \subset C([0, t], \mathbb{R}^n)$.
    Owing to \cref{assumption_K_2}, for every function $w(\cdot) \in \mathcal{W}[0, t]$, the function $\ell(\cdot) \in L_\infty([0, t], \mathbb{R}^n)$ for which representation \cref{representation_x} holds is determined uniquely (up to values taken on a zero Lebesgue measure subset of $[0, t]$), and we denote this function by $\ell(\cdot; t, w(\cdot))$.
    Then, we have
    \begin{equation} \label{representation_via_l_t_w}
        w(\tau)
        = y(\tau) + \int_{0}^{\tau} K(\tau, \xi) \ell(\xi; t, w(\cdot)) \, \rd \xi
        \quad \forall \tau \in [0, T], \ w(\cdot) \in \mathcal{W}[0, t].
    \end{equation}
    For $t = 0$, the set $\mathcal{W}[0, 0]$ is defined as consisting of the single function $w \colon [0, 0] \to \mathbb{R}^n$ such that $w(0) \doteq y(0)$.
    In this case, we formally put $\ell(0; 0, w(\cdot)) \doteq 0$ and observe that formula \cref{representation_via_l_t_w} takes place.

    Now, let $\mathcal{G}$ be the set of pairs $(t, w(\cdot))$ such that $t \in [0, T]$ and $w(\cdot) \in \mathcal{W}[0, t]$, i.e.,
    \begin{displaymath}
        \mathcal{G}
        \doteq \bigcup_{t \in [0, T]} \bigl( \{t\} \times \mathcal{W}[0, t] \bigr).
    \end{displaymath}
    We endow the set $\mathcal{G}$ with the metric
    \begin{displaymath}
        \dist \bigl( (t, w(\cdot)), (t^\prime, w^\prime(\cdot)) \bigr)
        \doteq |t - t^\prime| + \max_{\tau \in [0, T]} \|w(\tau \wedge t) - w^\prime(\tau \wedge t^\prime)\|,
    \end{displaymath}
    where $(t, w(\cdot))$, $(t^\prime, w^\prime(\cdot)) \in \mathcal{G}$ and $a \wedge b \doteq \min\{a, b\}$ for all $a$, $b \in \mathbb{R}$.
    By construction, for any $x(\cdot) \in \mathcal{W}[0, T]$ and $t \in [0, T]$, we get $x_t(\cdot) \in \mathcal{W}[0, t]$ and, hence, $(t, x_t(\cdot)) \in \mathcal{G}$.
    Moreover, it can be verified directly that the following mapping is continuous:
    \begin{equation} \label{mapping_basic}
        [0, T] \times C([0, T], \mathbb{R}^n) \supset [0, T] \times \mathcal{W}[0, T] \ni (t, x(\cdot)) \mapsto (t, x_t(\cdot)) \in \mathcal{G}.
    \end{equation}

    Note that, for any players' controls $u(\cdot) \in \mathcal{U}[0, T]$ and $v(\cdot) \in \mathcal{V}[0, T]$, the motion $x(\cdot) \doteq x(\cdot; u(\cdot), v(\cdot))$ of system \cref{system} satisfies the inclusion $x(\cdot) \in \mathcal{W}[0, T]$ and
    \begin{displaymath}
        \ell(\xi; T, x(\cdot))
        = f(\xi, x(\xi), u(\xi), v(\xi))
        \quad \text{for a.e. } \xi \in [0, T].
    \end{displaymath}
    Consequently, for every $t \in [0, T]$, the history $x_t(\cdot)$ of this motion on the time interval $[0, t]$ belongs to $\mathcal{W}[0, t]$, and, therefore, $(t, x_t(\cdot)) \in \mathcal{G}$.
    In this connection, $\mathcal{G}$ is treated as the {\it space of positions} of system \cref{system}.

    Note also that, in the two particular cases of the integral equation \cref{system} mentioned in \cref{section_statement}, we have $\ell(\xi; T, x(\cdot)) = \dot{x}(\xi)$ for a.e. $\xi \in [0, T]$ if \cref{y_constant,ordinary} hold and $\ell(\xi; T, x(\cdot)) = (^C D^{\boldsymbol{\alpha}} x)(\xi)$ for a.e. $\xi \in [0, T]$ if \cref{y_constant,fractional} are valid.
    In this sense, $\ell(\cdot; T, x(\cdot))$ may also be interpreted as a certain ``derivative'' of $x(\cdot)$.

    Further, let us extend the statement of the game \cref{system,cost_functional} to the case of an arbitrary {\it initial position} $(t, w(\cdot)) \in \mathcal{G}$.
    Define the set of {\it admissible extensions} $x(\cdot)$ of $w(\cdot)$ by
    \begin{equation} \label{X_t_w}
        \mathcal{X}(t, w(\cdot))
        \doteq \bigl\{ x(\cdot) \in \mathcal{W}[0, T] \colon
        x_t(\cdot) = w(\cdot) \bigr\}.
    \end{equation}
    Introduce the function
    \begin{equation} \label{a}
        a(\tau; t, w(\cdot))
        \doteq \begin{cases}
            w(\tau) & \forall \tau \in [0, t), \\
            \displaystyle
            y(\tau) + \int_{0}^{t} K(\tau, \xi) \ell(\xi; t, w(\cdot)) \, \rd \xi & \forall \tau \in [t, T].
          \end{cases}
    \end{equation}
    Observe that $a(\cdot; t, w(\cdot)) \in \mathcal{X}(t, w(\cdot)) \subset \mathcal{W}[0, T]$ and
    \begin{equation} \label{l_a}
        \ell \bigl( \xi; T, a(\cdot; t, w(\cdot)) \bigr)
        = \begin{cases}
            \ell(\xi; t, w(\cdot)) & \mbox{for a.e. } \xi \in [0, t), \\
            0 & \mbox{for a.e. } \xi \in [t, T].
          \end{cases}
    \end{equation}
    In particular, denoting $a^\prime(\cdot) \doteq a(\cdot; t, w(\cdot))$, we have the following equality:
    \begin{equation} \label{semigroup_property}
        a^\prime(\cdot)
        = a(\cdot; \tau, a^\prime_\tau(\cdot))
        \quad \forall \tau \in [t, T].
    \end{equation}
    Any functions $u(\cdot) \in \mathcal{U}[t, T]$ and $v(\cdot) \in \mathcal{V}[t, T]$ are players' {\it open-loop controls} on the time interval $[t, T]$.
    A {\it motion} of system \cref{system} generated from the position $(t, w(\cdot))$ by $u(\cdot) \in \mathcal{U}[t, T]$ and $v(\cdot) \in \mathcal{V}[t, T]$ is a function $x(\cdot) \in C([0, T], \mathbb{R}^n)$ satisfying the initial condition $x_t(\cdot) = w(\cdot)$ (see \cref{x_t}) and the Volterra integral equation
    \begin{equation} \label{system_arbitrary}
        x(\tau)
        = a(\tau; t, w(\cdot)) + \int_{t}^{\tau} K(\tau, \xi) f(\xi, x(\xi), u(\xi), v(\xi)) \, \rd \xi
        \quad \forall \tau \in [t, T].
    \end{equation}
    Equivalently, this motion can be defined as a function $x(\cdot) \in C([0, T], \mathbb{R}^n)$ such that
    \begin{displaymath}
        x(\tau)
        = y(\tau) + \int_{0}^{\tau} K(\tau, \xi) f^\prime(\xi, x(\xi)) \, \rd \xi
        \quad \forall \tau \in [0, T]
    \end{displaymath}
    with the function $f^\prime \colon [0, T] \times \mathbb{R}^n \to \mathbb{R}^n$ given by
    \begin{displaymath}
        f^\prime(\xi, x)
        \doteq \begin{cases}
            \ell(\xi; t, w(\cdot)) & \mbox{for a.e. } \xi \in [0, t), \\
            f(\xi, x, u(\xi), v(\xi)) & \forall \xi \in [t, T],
          \end{cases}
    \end{displaymath}
    where $x \in \mathbb{R}^n$.
    Therefore, arguing similarly to the proof of \cref{proposition_existence_uniqueness}, it can be shown that such a motion $x(\cdot) \doteq x(\cdot; t, w(\cdot), u(\cdot), v(\cdot))$ exists and is unique.
    Note also that, by construction, $x(\cdot) \in \mathcal{X}(t, w(\cdot)) \subset \mathcal{W}[0, T]$ and
    \begin{equation} \label{motion_arbitrary_ell}
        \ell(\xi; T, x(\cdot))
        = \begin{cases}
            \ell(\xi; t, w(\cdot)) & \mbox{for a.e. } \xi \in [0, t), \\
            f(\xi, x(\xi), u(\xi), v(\xi)) & \mbox{for a.e. } \xi \in [t, T].
          \end{cases}
    \end{equation}
    Let us consider the {\it game} in which the first player tries to minimize while the second player tries to maximize the {\it cost functional}
    \begin{equation} \label{cost_functional_arbitrary}
        J(t, w(\cdot), u(\cdot), v(\cdot))
        \doteq \sigma(x(\cdot))
        + \int_{t}^{T} \chi(\tau, x(\tau), u(\tau), v(\tau)) \, \rd \tau,
    \end{equation}
    where $u(\cdot) \in \mathcal{U}[t, T]$, $v(\cdot) \in \mathcal{V}[t, T]$, and $x(\cdot) \doteq x(\cdot; t, w(\cdot), u(\cdot), v(\cdot))$.
    By analogy with \cref{section_statement}, let us introduce the {\it lower} and {\it upper values} of this game by
    \begin{equation} \label{lower_upper_value_functionals}
        \begin{aligned}
            \rho_-(t, w(\cdot))
            & \doteq \inf_{\boldsymbol{a} \in \boldsymbol{\mathcal{A}}[t, T]} \sup_{v(\cdot) \in \mathcal{V}[t, T]}
            J \bigl(t, w(\cdot), \boldsymbol{a}[v(\cdot)](\cdot), v(\cdot) \bigr), \\
            \rho_+(t, w(\cdot))
            & \doteq \sup_{\boldsymbol{b} \in \boldsymbol{\mathcal{B}}[t, T]} \inf_{u(\cdot) \in \mathcal{U}[t, T]} J \bigl(t, w(\cdot), u(\cdot), \boldsymbol{b}[u(\cdot)](\cdot) \bigr).
        \end{aligned}
    \end{equation}

    Relations \cref{lower_upper_value_functionals} actually define the functionals $\rho_- \colon \mathcal{G} \to \mathbb{R}$ and $\rho_+ \colon \mathcal{G} \to \mathbb{R}$, which are called the {\it lower} and {\it upper value functionals} respectively.

    In the particular case where $t = 0$, we have $x(\cdot; u(\cdot), v(\cdot)) = x(\cdot; 0, w(\cdot), u(\cdot), v(\cdot))$ for all $u(\cdot) \in \mathcal{U}[0, T]$ and $v(\cdot) \in \mathcal{V}[0, T]$.
    Hence, the statement of the game given above indeed extends the original statement from \cref{section_statement} and the equalities below hold:
    \begin{equation} \label{values_0}
        \rho_-(0, w(\cdot))
        = \rho_-^0,
        \quad \rho_+(0, w(\cdot))
        = \rho_+^0.
    \end{equation}
    So, in order to prove the existence of the game value (see \cref{game_value}), we will show that the lower and upper value functionals coincide.
    To this end, we will prove that each of these functionals coincides with a unique viscosity solution of a Cauchy problem for a certain path-dependent Hamilton--Jacobi equation.

\section{Semigroup property and dynamic programming principle}
\label{section_DPP}

    The next proposition establishes the semigroup property of system motions.
    \begin{proposition} \label{proposition_semigroup_property}
        Let $(t, w(\cdot)) \in \mathcal{G}$, $u(\cdot) \in \mathcal{U}[t, T]$, and $v(\cdot) \in \mathcal{V}[t, T]$ be fixed and let $x(\cdot) \doteq x(\cdot; t, w(\cdot), u(\cdot), v(\cdot))$ be the corresponding motion of system \cref{system}.
        Suppose that $t^\prime \in [t, T]$, $u^\prime(\cdot) \in \mathcal{U}[t^\prime, T]$, and $v^\prime(\cdot) \in \mathcal{V}[t^\prime, T]$ are given and consider the system motion $x^\prime(\cdot) \doteq x(\cdot; t^\prime, x_{t^\prime}(\cdot), u^\prime(\cdot), v^\prime(\cdot))$.
        Then, $x^\prime(\cdot)$ is the system motion generated from the position $(t, w(\cdot))$ by the players' controls
        \begin{displaymath}
            u^{\prime \prime}(\tau)
            \doteq
            \begin{cases}
                u(\tau) & \forall \tau \in [t, t^\prime), \\
                u^\prime(\tau) & \forall \tau \in [t^\prime, T],
            \end{cases}
            \quad v^{\prime \prime}(\tau)
            \doteq
            \begin{cases}
                v(\tau) & \forall \tau \in [t, t^\prime), \\
                v^\prime(\tau) & \forall \tau \in [t^\prime, T].
            \end{cases}
        \end{displaymath}
        Namely, it holds that $x^\prime(\cdot) = x(\cdot; t, w(\cdot), u^{\prime \prime}(\cdot), v^{\prime \prime}(\cdot))$.
    \end{proposition}

    The proof is straightforward due to the results of the previous \cref{section_set_of_positions}.

    Note that, under the conditions of \cref{proposition_semigroup_property}, we obtain (see \cref{cost_functional_arbitrary})
    \begin{displaymath}
        J(t, w(\cdot), u^{\prime \prime}(\cdot), v^{\prime \prime}(\cdot))
        = J(t^\prime, x_{t^\prime}(\cdot), u^\prime(\cdot), v^\prime(\cdot))
        + \int_{t}^{t^\prime} \chi (\tau, x(\tau), u(\tau), v(\tau)) \, \rd \tau.
    \end{displaymath}

    Taking this relation into account and repeating the proof of \cite[Theorem 3.1]{Evans_Souganidis_1984} (see also, e.g., \cite[Chapter XI, Theorem 5.1]{Fleming_Soner_2006} and \cite[Theorem 3.3.5]{Yong_2015}), it can be verified that the lower $\rho_-$ and upper $\rho_+$ value functionals (see \cref{lower_upper_value_functionals}) have the following properties, expressing the dynamic programming principle in the game under consideration.
    \begin{proposition} \label{proposition_DPP}
        Let $(t, w(\cdot)) \in \mathcal{G}$ and $\theta \in [t, T]$.
        Then,
        \begin{displaymath}
            \rho_-(t, w(\cdot)) = \inf_{\boldsymbol{a} \in \boldsymbol{\mathcal{A}}[t, T]} \sup_{v(\cdot) \in \mathcal{V}[t, T]}
            \biggl( \rho_-(\theta, x_\theta(\cdot)) + \int_{t}^{\theta} \chi \bigl( \tau, x(\tau), \boldsymbol{a}[v(\cdot)](\tau), v(\tau) \bigr) \, \rd \tau \biggr),
        \end{displaymath}
        where $x(\cdot) \doteq x(\cdot; t, w(\cdot), \boldsymbol{a}[v(\cdot)](\cdot), v(\cdot))$, and, analogously,
        \begin{displaymath}
            \rho_+(t, w(\cdot)) = \sup_{\boldsymbol{b} \in \boldsymbol{\mathcal{B}}[t, T]} \inf_{u(\cdot) \in \mathcal{U}[t, T]}
            \biggl( \rho_+(\theta, x_\theta(\cdot)) + \int_{t}^{\theta} \chi \bigl( \tau, x(\tau), u(\tau), \boldsymbol{b}[u(\cdot)](\tau) \bigr) \, \rd \tau \biggr),
        \end{displaymath}
        where $x(\cdot) \doteq x(\cdot; t, w(\cdot), u(\cdot), \boldsymbol{b}[u(\cdot)](\cdot))$.
    \end{proposition}

\section{Auxiliary sequence of compact sets}
\label{section_sequence}

    For every $k \in \mathbb{N}$, consider the set
    \begin{equation} \label{X_k}
        \mathcal{X}_k
        \doteq \bigl\{ x(\cdot) \in \mathcal{W}[0, T] \colon
        \| \ell(\xi; T, x(\cdot)) \|
        \leq k c (1 + \|x(\xi)\|) \text{ for a.e. } \xi \in [0, T] \bigr\},
    \end{equation}
    where the function $\ell(\cdot; T, x(\cdot))$ is determined by the function $x(\cdot)$ according to \cref{representation_via_l_t_w} and $c$ is taken from \cref{assumption_f}, (c).
    Since $y(\cdot) \in \mathcal{W}[0, T]$ and $\ell(\xi; T, y(\cdot)) = 0$ for a.e. $\xi \in [0, T]$, we get $y(\cdot) \in \mathcal{X}_k$, and, therefore, $\mathcal{X}_k$ is not empty.

    \begin{proposition} \label{proposition_X_k}
        For every $k \in \mathbb{N}$, the set $\mathcal{X}_k$ is compact in $C([0, T], \mathbb{R}^n)$.
    \end{proposition}
    \begin{proof}
        Take $\kappa_\ast$ from \cref{mu_ast} and put
        \begin{displaymath}
            R_k
            \doteq (1 + \|y(\cdot)\|_{C([0, T], \mathbb{R}^n)}) \mathrm{E}_\alpha(\mathrm{\Gamma}(\alpha) \kappa_\ast k c T^\alpha) - 1,
            \quad M_k
            \doteq k c (1 + R_k).
        \end{displaymath}
        Due to \cref{representation_via_l_t_w,X_k}, for every $x(\cdot) \in \mathcal{X}_k$, we have
        \begin{displaymath}
            \begin{aligned}
                \|x(\tau)\|
                & \leq \|y(\tau)\| + \int_{0}^{\tau} \|K(\tau, \xi)\| \|\ell(\xi; T, x(\cdot))\| \, \rd \xi \\
                & \leq \|y(\cdot)\|_{C([0, T], \mathbb{R}^n)}
                + \kappa_\ast k c \int_{0}^{\tau} \frac{1 + \|x(\xi)\|}{(\tau - \xi)^{1 - \alpha}} \, \rd \xi
                \quad \forall \tau \in [0, T].
            \end{aligned}
        \end{displaymath}
        Therefore, applying the generalized Gronwall inequality (see, e.g., \cite[Lemma 1.3.13]{Brunner_2017}), we derive $\|x(\cdot)\|_{C([0, T], \mathbb{R}^n)} \leq R_k$ and, consequently, $\|\ell(\cdot; T, x(\cdot))\|_{L_\infty([0, T], \mathbb{R}^n)} \leq M_k$.
        Then, recalling definition \cref{operator_K} of the operator $\mathbf{K}_\infty$ and denoting
        \begin{equation} \label{E_k}
            E_k
            \doteq \bigl\{ \ell(\cdot) \in L_\infty([0, T], \mathbb{R}^n) \colon
            \|\ell(\cdot)\|_{L_\infty([0, T], \mathbb{R}^n)}
            \leq M_k \bigr\},
        \end{equation}
        we find that $\mathcal{X}_k \subset \{y(\cdot)\} + \mathbf{K}_\infty[E_k]$.
        Hence, compactness of $\mathbf{K}_\infty$ implies that $\mathcal{X}_k$ is relatively compact in $C([0, T], \mathbb{R}^n)$, and, thus, it remains to prove that $\mathcal{X}_k$ is closed.

        Let $\{x^{(i)}(\cdot)\}_{i = 1}^\infty \subset \mathcal{X}_k$, $x^\ast(\cdot) \in C([0, T], \mathbb{R}^n)$, and $\|x^{(i)}(\cdot) - x^\ast(\cdot)\|_{C([0, T], \mathbb{R}^n)} \to 0$ as $i \to \infty$.
        Denote $\ell^{(i)}(\cdot) \doteq \ell(\cdot; T, x^{(i)}(\cdot))$, $i \in \mathbb{N}$.
        Fix $p \in (1 / \alpha, \infty)$ and observe that the set $E_k$ is weakly sequentially compact as a subset of $L_p([0, T], \mathbb{R}^n)$ (see, e.g., the proof of \cite[Lemma 4.1]{Gomoyunov_2018_FCAA}).
        So, we may assume, by extracting a subsequence if necessary, that the sequence $\{\ell^{(i)}(\cdot)\}_{i = 1}^\infty$ converges weakly in $L_p([0, T], \mathbb{R}^n)$ to some function $\overline{\ell}(\cdot) \in E_k$.
        In view of compactness of the operator $\mathbf{K}_p$ from \cref{operator_K}, we get
        \begin{displaymath}
            \|x^{(i)}(\cdot) - y(\cdot) - \mathbf{K}_p [\overline{\ell}(\cdot)](\cdot)\|_{C([0, T], \mathbb{R}^n)}
            = \|\mathbf{K}_p [\ell^{(i)}(\cdot)](\cdot) - \mathbf{K}_p [\overline{\ell}(\cdot)](\cdot)\|_{C([0, T], \mathbb{R}^n)}
            \to 0
        \end{displaymath}
        as $i \to \infty$.
        Consequently, $x^\ast(\cdot) = y(\cdot) + \mathbf{K}_p [\overline{\ell}(\cdot)](\cdot)$, which yields  $x^\ast(\cdot) \in \mathcal{W}[0, T]$ and $\overline{\ell}(\cdot) = \ell(\cdot; T, x^\ast(\cdot))$.

        Fix $\varepsilon > 0$ and choose $i_\ast \in \mathbb{N}$ such that $\|x^{(i)}(\cdot) - x^\ast(\cdot)\|_{C([0, T], \mathbb{R}^n)} \leq \varepsilon$ for all $i \geq i_\ast$.
        Then, for every $i \geq i_\ast$, we obtain
        \begin{displaymath}
            \|\ell^{(i)}(\xi)\|
            \leq k c (1 + \|x^{(i)}(\xi)\|)
            \leq k c (1 + \|x^\ast(\xi)\| + \varepsilon)
            \quad \text{for a.e. } \xi \in [0, T].
        \end{displaymath}
        Let $j \in \mathbb{N}$.
        Since $\{\ell^{(i)}(\cdot)\}_{i = i_\ast}^\infty$ converges weakly in $L_p([0, T], \mathbb{R}^n)$ to $\overline{\ell}(\cdot)$, there exists (see, e.g., \cite[Theorem 3.13]{Rudin_1991}) a convex combination
        \begin{equation} \label{convex_combination_1}
            m^{(j)}(\cdot)
            \doteq \sum_{p = 1}^{q_j} \gamma_{p, j} \ell^{(i_{p, j})}(\cdot)
        \end{equation}
        such that $\|m^{(j)}(\cdot) - \overline{\ell}(\cdot)\|_{L_p([0, T], \mathbb{R}^n)} \leq 1 / j$.
        Here, $q_j \in \mathbb{N}$, $i_{p, j} \geq i_\ast$ and $\gamma_{p, j} \in [0, 1]$ for all $p \in \overline{1, q_j}$, and $\sum_{p = 1}^{q_j} \gamma_{p, j} = 1$.
        Note that
        \begin{equation} \label{m_j_estimate}
            \|m^{(j)}(\xi)\|
            \leq \sum_{p = 1}^{q_j} \gamma_{p, j} \|\ell^{(i_{p, j})}(\xi)\|
            \leq k c (1 + \|x^\ast(\xi)\| + \varepsilon)
            \quad \text{for a.e. } \xi \in [0, T].
        \end{equation}
        Due to the convergence $\|m^{(j)}(\cdot) - \overline{\ell}(\cdot)\|_{L_p([0, T], \mathbb{R}^n)} \to 0$ as $j \to \infty$, we may assume, by extracting a subsequence if necessary, that $\|m^{(j)}(\xi) - \overline{\ell}(\xi)\|$ as $j \to \infty$ for a.e. $\xi \in [0, T]$.
        Hence, letting $j \to \infty$ in \cref{m_j_estimate}, we derive $\|\overline{\ell}(\xi)\| \leq k c (1 + \|x^\ast(\xi)\| + \varepsilon)$ for a.e. $\xi \in [0, T]$.
        Using the fact that $\varepsilon > 0$ was taken arbitrarily, we eventually get
        \begin{displaymath}
            \|\ell(\xi; T, x^\ast(\cdot))\|
            = \|\overline{\ell}(\xi)\|
            \leq k c (1 + \|x^\ast(\xi)\|)
            \quad \text{for a.e. } \xi \in [0, T],
        \end{displaymath}
        which implies that $x^\ast(\cdot) \in \mathcal{X}_k$ and completes the proof.
    \end{proof}

    In addition, note that, for any $x(\cdot) \in \mathcal{W}[0, T]$, there exists $k \in \mathbb{N}$ such that
    \begin{displaymath}
        \|\ell(\xi; T, x(\cdot))\|
        \leq \|\ell(\cdot; T, x(\cdot))\|_{L_\infty([0, T], \mathbb{R}^n)}
        \leq k c
        \leq k c (1 + \|x(\xi)\|)
    \end{displaymath}
    for a.e. $\xi \in [0, T]$, and, therefore $x(\cdot) \in \mathcal{X}_k$.
    This means that
    \begin{equation} \label{W_is_union}
        \mathcal{W}[0, T]
        = \bigcup_{k \in \mathbb{N}} \mathcal{X}_k.
    \end{equation}

    Further, for every $k \in \mathbb{N}$, consider the set
    \begin{equation} \label{G_k}
        \mathcal{G}_k
        \doteq \bigl\{ (t, w(\cdot)) \in \mathcal{G} \colon
        \|\ell(\xi; t, w(\cdot))\|
        \leq k c (1 + \|w(\xi)\|) \text{ for a.e. } \xi \in [0, t] \bigr\}.
    \end{equation}
    In other words, $\mathcal{G}_k$ is the image of the set $[0, T] \times \mathcal{X}_k$ under mapping \cref{mapping_basic}.
    Then, owing to continuity of this mapping, we conclude that $\mathcal{G}_k$ is a compact subset of $\mathcal{G}$.
    Moreover, the union of $\mathcal{G}_k$ over all $k \in \mathbb{N}$ coincides with $\mathcal{G}$ by \cref{W_is_union}.
    Finally, each of the sets $\mathcal{G}_k$ is {\it invariant} with respect to motions of the dynamical system \cref{system} due to \cref{assumption_f}, (c).
    Namely, for any initial position $(t, w(\cdot)) \in \mathcal{G}_k$ and any players' controls $u(\cdot) \in \mathcal{U}[t, T]$, $v(\cdot) \in \mathcal{V}[t, T]$, we have $x(\cdot) \doteq x(\cdot; t, w(\cdot), u(\cdot), v(\cdot)) \in \mathcal{X}_k$ and, consequently, $(\tau, x_\tau(\cdot)) \in \mathcal{G}_k$ for all $\tau \in [0, T]$.

\section{Continuity properties of lower and upper value functionals}
\label{section_continuity}

    Before studying continuity properties of the functionals $\rho_-$ and $\rho_+$ (see \cref{lower_upper_value_functionals}), we establish the following result on continuity of the free term of the integral equation \cref{system_arbitrary}.
    \begin{lemma} \label{lemma_a_mapping}
        For every $k \in \mathbb{N}$, the mapping below is (uniformly) continuous:
        \begin{displaymath}
            \mathcal{G}_k \ni (t, w(\cdot))
            \mapsto a(\cdot; t, w(\cdot)) \in C([0, T], \mathbb{R}^n).
        \end{displaymath}
    \end{lemma}
    \begin{proof}
        Since $a(\cdot; t, w(\cdot)) \in \mathcal{X}_k$ for all $(t, w(\cdot)) \in \mathcal{G}_k$ (see \cref{l_a}) and $\mathcal{X}_k$ is compact by \cref{proposition_X_k}, in order to prove the result, it suffices to fix $\{(t^{(i)}, w^{(i)}(\cdot))\}_{i = 1}^\infty \subset \mathcal{G}_k$ and $(t^\ast, w^\ast(\cdot)) \in \mathcal{G}_k$ such that, first, $\dist((t^{(i)}, w^{(i)}(\cdot)), (t^\ast, w^\ast(\cdot))) \to 0$ as $i \to \infty$ and, second, $\|a(\cdot; t^{(i)}, w^{(i)}(\cdot)) - \tilde{a}(\cdot)\|_{C([0, T], \mathbb{R}^n)} \to 0$ as $i \to \infty$ for some function $\tilde{a}(\cdot) \in \mathcal{X}_k$ and verify that $\tilde{a}(\cdot) = a(\cdot; t^\ast, w^\ast(\cdot))$.
        Denote $a^{(i)}(\cdot) \doteq a(\cdot; t^{(i)}, w^{(i)}(\cdot))$, $i \in \mathbb{N}$.
        Taking into account that $|t^{(i)} - t^\ast| \to 0$ and $\|a^{(i)}(\cdot) - \tilde{a}(\cdot)\|_{C([0, T], \mathbb{R}^n)} \to 0$ as $i \to \infty$ and mapping \cref{mapping_basic} is continuous, we derive
        \begin{displaymath}
            \dist \bigl( (t^{(i)}, a^{(i)}_{t^{(i)}}(\cdot)), (t^\ast, \tilde{a}_{t^\ast}(\cdot)) \bigr)
            \to 0
            \quad \text{as } i \to \infty.
        \end{displaymath}
        Hence, recalling that $a^{(i)}_{t^{(i)}}(\cdot) = w^{(i)}(\cdot)$ for all $i \in \mathbb{N}$ (see \cref{a}), we get $\tilde{a}_{t^\ast}(\cdot) = w^\ast(\cdot)$ and, consequently, $\tilde{a}(\tau) = a(\tau; t^\ast, w^\ast(\cdot))$, $\tau \in [0, t^\ast]$.
        Thus, the proof is complete if $t^\ast = T$.
        In the case where $t^\ast < T$, it remains to verify that $\ell(\xi; T, \tilde{a}(\cdot)) = 0$ for a.e. $\xi \in [t^\ast, T]$ (see \cref{l_a}).
        Denote $\ell^{(i)}(\cdot) \doteq \ell(\cdot; T, a^{(i)}(\cdot))$, $i \in \mathbb{N}$.
        For every $i \in \mathbb{N}$, in view of the inclusion $a^{(i)}(\cdot) \in \mathcal{X}_k$, we have $\ell^{(i)}(\cdot) \in E_k$, where the set $E_k$ is given by \cref{E_k}.
        Then, taking $p \in (1 / \alpha, \infty)$ and arguing similarly to the proof of \cref{proposition_X_k}, we may assume that the sequence $\{\ell^{(i)}(\cdot)\}_{i = 1}^\infty$ converges weakly in $L_p([0, T], \mathbb{R}^n)$ to some function $\overline{\ell}(\cdot) \in E_k$.
        In particular, this yields (see \cref{operator_K,representation_via_l_t_w})
        \begin{displaymath}
            \|a^{(i)}(\cdot) - y(\cdot) - \mathbf{K}_p [\overline{\ell}(\cdot)](\cdot)\|_{C([0, T], \mathbb{R}^n)}
            = \|\mathbf{K}_p [\ell^{(i)}(\cdot)](\cdot) - \mathbf{K}_p [\overline{\ell}(\cdot)](\cdot)\|_{C([0, T], \mathbb{R}^n)}
            \to 0
        \end{displaymath}
        as $i \to \infty$, wherefrom it follows that $\overline{\ell}(\cdot) = \ell(\cdot; T, \tilde{a}(\cdot))$.
        Hence, we need to show that $\overline{\ell}(\xi) = 0$ for a.e. $\xi \in [t^\ast, T]$.
        To this end, we fix $\delta \in (0, T - t^\ast)$ and prove that $\overline{\ell}(\xi) = 0$ for a.e. $\xi \in [t^\ast + \delta, T]$.
        Since $|t^{(i)} - t^\ast| \to 0$ as $i \to \infty$, there exists $i_\ast \in \mathbb{N}$ such that $t^{(i)} \leq t^\ast + \delta$ for all $i \geq i_\ast$.
        Note that, for every $i \geq i_\ast$ and a.e. $\xi \in [t^\ast + \delta, T]$, we have $\ell^{(i)}(\xi) = 0$ (see \cref{l_a}).
        Further, arguing again as in the proof of \cref{proposition_X_k}, we can find a sequence of convex combinations $\{m^{(j)}(\cdot)\}_{j = 1}^\infty$ of form \cref{convex_combination_1} such that $\|m^{(j)}(\xi) - \overline{\ell}(\xi)\| \to 0$ as $j \to \infty$ for a.e. $\xi \in [0, T]$.
        By construction, $m^{(j)}(\xi) = 0$ for every $j \in \mathbb{N}$ and a.e. $\xi \in [t^\ast + \delta, T]$.
        Therefore, we conclude that $\overline{\ell}(\xi) = 0$ for a.e. $\xi \in [t^\ast + \delta, T]$ and complete the proof.
    \end{proof}

    Denote by $\Phi$ the set of functionals $\varphi \colon \mathcal{G} \to \mathbb{R}$ having the properties below.

    \noindent (a)
        For every $k \in \mathbb{N}$, the restriction of $\varphi$ to $\mathcal{G}_k$ is (uniformly) continuous.

    \noindent (b)
        For every $k \in \mathbb{N}$, there exists $\lambda > 0$ such that, for any $(t, w(\cdot))$, $(t, w^\prime(\cdot)) \in \mathcal{G}_k$,
        \begin{equation} \label{varphi.2_inequality}
            \begin{aligned}
                & |\varphi(t, w(\cdot)) - \varphi(t, w^\prime(\cdot))| \\
                & \ \ \leq \lambda \biggl( \|a(T; t, w(\cdot)) - a(T; t, w^\prime(\cdot))\|
                + \int_{0}^{T} \frac{\|a(\tau; t, w(\cdot)) - a(\tau; t, w^\prime(\cdot))\|}{(T - \tau)^{1 - \alpha}} \, \rd \tau \biggr).
            \end{aligned}
        \end{equation}

    \begin{lemma} \label{lemma_continuity}
        The inclusions $\rho_-$, $\rho_+ \in \Phi$ take place.
    \end{lemma}

    Indeed, the proof of the fact that $\rho_-$ and $\rho_+$ have property (b) repeats the proof of \cite[Lemma 1]{Gomoyunov_Lukoyanov_2021} with insignificant changes only.
    Note that this is the place where \cref{assumption_sigma_chi}, (c), is important.
    The fact that $\rho_-$ and $\rho_+$ have property (a) can be proved by rather standard arguments (see, e.g., \cite[Chapter XI, Theorem 5.2]{Fleming_Soner_2006}) if we take \cref{proposition_DPP}, property (b), and \cref{lemma_a_mapping} into account.

\section{Coinvariant derivatives}
\label{section_ci_derivatives}

    Denote $\mathcal{G}^\circ \doteq \{(t, w(\cdot)) \in \mathcal{G} \colon t < T\}$.

    A functional $\varphi \colon \mathcal{G} \to \mathbb{R}$ is called {\it coinvariantly differentiable} ({\it $ci$-differentiable} for short) at a point $(t, w(\cdot)) \in \mathcal{G}^\circ$ if there exist $\boldsymbol{\partial}_t \varphi (t, w(\cdot)) \in \mathbb{R}$ and $\boldsymbol{\nabla} \varphi (t, w(\cdot)) \in \mathbb{R}^n$ such that, for every function $x(\cdot) \in \mathcal{X}(t, w(\cdot))$ (see \cref{X_t_w}),
    \begin{equation} \label{ci_derivatives}
        \begin{aligned}
            & \biggl| \frac{\varphi(\tau, x_\tau(\cdot)) - \varphi(t, w(\cdot))}{\tau - t} \\
            & \quad - \boldsymbol{\partial}_t \varphi(t, w(\cdot))
            - \biggl\langle \boldsymbol{\nabla} \varphi(t, w(\cdot)), \frac{1}{\tau - t} \int_{t}^{\tau} \ell(\xi; T, x(\cdot)) \, \rd \xi \biggr\rangle \biggr|
            \to 0
            \quad \text{as } \tau \to t^+,
        \end{aligned}
    \end{equation}
    where $x_\tau(\cdot)$ is the restriction of $x(\cdot)$ to $[0, \tau]$ (see \cref{x_t}) and $\ell(\cdot; T, x(\cdot))$ is determined by $x(\cdot)$ according to \cref{representation_via_l_t_w}.
    In this case, $\boldsymbol{\partial}_t \varphi (t, w(\cdot))$ and $\boldsymbol{\nabla} \varphi(t, w(\cdot))$ are called the {\it $ci$-derivatives} of the functional $\varphi$ at the point $(t, w(\cdot))$; note that they are determined by relation \cref{ci_derivatives} uniquely.

    Let us emphasize that the introduced notion of $ci$-differentiability depends on the kernel $K$ from the integral equation \cref{system}.
    In particular, if we suppose \cref{y_constant} and take kernel \cref{ordinary}, we come to the usual notion of $ci$-differentiability (see, e.g., \cite{Kim_1999,Lukoyanov_2000_PMM_Eng} and also \cite[Section 5.2]{Gomoyunov_Lukoyanov_Plaksin_2021});
    and if we suppose \cref{y_constant} and consider kernel \cref{fractional}, we obtain a generalization of the notion of fractional $ci$-differentiability (see, e.g., \cite{Gomoyunov_2020_SIAM}) to the case of multi-order $\boldsymbol{\alpha}$.

    Further, a functional $\varphi \colon \mathcal{G} \to \mathbb{R}$ is called {\it $ci$-smooth} if it is $ci$-differentiable at every point $(t, w(\cdot)) \in \mathcal{G}^\circ$ and the restriction of $\varphi$ to $\mathcal{G}_k$ (see \cref{G_k}) as well as the restrictions of $\boldsymbol{\partial}_t \varphi$ and $\boldsymbol{\nabla} \varphi$ to $\mathcal{G}_k \cap \mathcal{G}^\circ$ are continuous for every $k \in \mathbb{N}$.

    The proposition below provides the key property of $ci$-smooth functionals and allows us to calculate their total derivatives along motions of system \cref{system}.
    \begin{proposition} \label{proposition_ci_smooth}
        Given a $ci$-smooth functional $\varphi \colon \mathcal{G} \to \mathbb{R}$ and $x(\cdot) \in \mathcal{W}[0, T]$, consider the function $\omega(\tau) \doteq \varphi(\tau, x_\tau(\cdot))$, $\tau \in [0, T]$.
        Then, $\omega(\cdot)$ is continuous;
        $\omega(\cdot)$ is Lipschitz continuous on $[0, \theta]$ for every $\theta \in (0, T)$;
        and
        \begin{equation} \label{proposition_ci_smooth_main}
            \dot{\omega}(\tau)
            = \boldsymbol{\partial}_t \varphi (\tau, x_\tau(\cdot))
            + \langle \boldsymbol{\nabla} \varphi(\tau, x_\tau(\cdot)), \ell(\tau; T, x(\cdot)) \rangle
            \quad \text{for a.e. } \tau \in [0, T],
        \end{equation}
        where $\ell(\cdot; T, x(\cdot))$ is determined by $x(\cdot)$ according to \cref{representation_via_l_t_w}.
    \end{proposition}
    \begin{proof}
        In view of \cref{W_is_union}, there exists $k \in \mathbb{N}$ such that $x(\cdot) \in \mathcal{X}_k$ and, consequently, $(\tau, x_\tau(\cdot)) \in \mathcal{G}_k$, $\tau \in [0, T]$.
        Therefore, and due to continuity of mapping \cref{mapping_basic}, the functions $\omega(\cdot)$, $[0, T) \ni \tau \mapsto \boldsymbol{\partial}_t \varphi (\tau, x_\tau(\cdot)) \in \mathbb{R}$, and $[0, T) \ni \tau \mapsto \boldsymbol{\nabla} \varphi(\tau, x_\tau(\cdot)) \in \mathbb{R}^n$ are continuous.

        Let $\theta \in (0, T)$.
        Take $M > 0$ such that $|\boldsymbol{\partial}_t \varphi (\tau, x_\tau(\cdot))| \leq M$, $\|\boldsymbol{\nabla} \varphi(\tau, x_\tau(\cdot))\| \leq M$, $\tau \in [0, \theta]$.
        Denote $\ell(\cdot) \doteq \ell(\cdot; T, x(\cdot))$, $L \doteq (1 + \|\ell(\cdot)\|_{L_\infty([0, T], \mathbb{R}^n)}) M$.
        Let us show that
        \begin{equation} \label{proof_Lem_formula_omega}
            |\omega(\tau) - \omega(\tau^\prime)|
            \leq L |\tau - \tau^\prime|
            \quad \forall \tau, \tau^\prime \in [0, \theta].
        \end{equation}
        Fix $\tau \in [0, \theta]$ and $\varepsilon > 0$.
        Since $\varphi$ is $ci$-differentiable at $(\tau, x_\tau(\cdot))$ and $x(\cdot) \in \mathcal{X}(\tau, x_\tau(\cdot))$, then, according to \cref{ci_derivatives}, there exists $\delta_\ast \in (0, T - \tau]$ such that
        \begin{displaymath}
            \biggl| \frac{\omega(\tau^\prime) - \omega(\tau)}{\tau^\prime - \tau} - \boldsymbol{\partial}_t \varphi(\tau, x_\tau(\cdot))
            - \bigg\langle \boldsymbol{\nabla} \varphi(\tau, x_\tau(\cdot)), \frac{1}{\tau^\prime - \tau} \int_{\tau}^{\tau^\prime} \ell(\xi) \, \rd \xi \bigg\rangle \biggr|
            \leq \varepsilon
        \end{displaymath}
        for all $\tau^\prime \in (\tau, \tau + \delta_\ast]$.
        Hence, we have
        \begin{displaymath}
            \biggl| \limsup_{\tau^\prime \to \tau^+} \frac{\omega(\tau^\prime) - \omega(\tau)}{\tau^\prime - \tau} \biggr|
            \leq L + \varepsilon,
        \end{displaymath}
        wherefrom, letting $\varepsilon \to 0^+$, we derive the inequality
        \begin{displaymath}
            \biggl| \limsup_{\tau^\prime \to \tau^+} \frac{\omega(\tau^\prime) - \omega(\tau)}{\tau^\prime - \tau} \biggr|
            \leq L.
        \end{displaymath}
        Taking into account that this inequality holds for every $\tau \in [0, \theta]$ and $\omega(\cdot)$ is continuous on $[0, \theta]$, we get \cref{proof_Lem_formula_omega} by the Dini theorem (see, e.g., \cite[Chapter 4, Theorem 1.2]{Bruckner_1978}).

        Now, let $\mathcal{T}$ be the set of $\tau \in (0, T)$ such that the derivative $\dot{\omega}(\tau)$ exists and $\tau$ is a Lebesgue point of the function $\ell(\cdot)$, which yields
        \begin{displaymath}
            \lim_{\tau^\prime \to \tau^+} \frac{1}{\tau^\prime - \tau} \int_{\tau}^{\tau^\prime} \|\ell(\xi) - \ell(\tau)\| \, \rd \xi
            = 0.
        \end{displaymath}
        Then, for every $\tau \in \mathcal{T}$, we obtain
        \begin{displaymath}
            \begin{aligned}
                \dot{\omega}(\tau)
                & = \lim_{\tau^\prime \to \tau^+} \frac{\omega(\tau^\prime) - \omega(\tau)}{\tau^\prime - \tau} \\
                & = \boldsymbol{\partial}_t \varphi(\tau, x_\tau(\cdot))
                + \biggl\langle \boldsymbol{\nabla} \varphi(\tau, x_\tau(\cdot)), \lim_{\tau^\prime \to \tau^+} \frac{1}{\tau^\prime - \tau} \int_{\tau}^{\tau^\prime} \ell(\xi) \, \rd \xi \biggr\rangle \\
                & = \boldsymbol{\partial}_t \varphi(\tau, x_\tau(\cdot))
                + \langle \boldsymbol{\nabla} \varphi(\tau, x_\tau(\cdot)), \ell(\tau) \rangle.
            \end{aligned}
        \end{displaymath}
        As a result, observing that the set $[0, T] \setminus \mathcal{T}$ has zero Lebesgue measure, we derive relation \cref{proposition_ci_smooth_main}, which completes the proof.
    \end{proof}

\section{Path-dependent Hamilton--Jacobi equations and viscosity solutions}
\label{section_HJ}

    Let us consider the {\it Cauchy problem} for the {\it path-dependent Hamilton--Jacobi equation} with $ci$-derivatives (see \cref{ci_derivatives})
    \begin{equation} \label{HJ}
        \boldsymbol{\partial}_t \varphi(t, w(\cdot))
        + H \bigl( t, w(t), \boldsymbol{\nabla} \varphi(t, w(\cdot)) \bigr)
        = 0
        \quad \forall (t, w(\cdot)) \in \mathcal{G}^\circ
    \end{equation}
    under the right-end {\it boundary condition}
    \begin{equation} \label{boundary_condition}
        \varphi(T, w(\cdot))
        = \sigma(w(\cdot))
        \quad \forall w(\cdot) \in \mathcal{W}[0, T].
    \end{equation}
    The {\it unknown} is a functional $\varphi \colon \mathcal{G} \to \mathbb{R}$;
    the {\it Hamiltonian} $H \colon [0, T] \times \mathbb{R}^n \times \mathbb{R}^n \to \mathbb{R}$ and the {\it boundary functional} $\sigma \colon C([0, T], \mathbb{R}^n) \to \mathbb{R}$ are given.

    Following \cite{Crandall_Lions_1983,Crandall_Evans_Lions_1984} (see also, e.g., \cite{Soner_1988,Lukoyanov_2007_IMM_Eng,Gomoyunov_2023_JDE}), we define a {\it viscosity solution} of the Cauchy problem \cref{HJ,boundary_condition} as a functional $\varphi \in \Phi$ (see \cref{section_continuity}) that satisfies the boundary condition \cref{boundary_condition} and has the following two properties.

    \noindent (a)
        For any $ci$-smooth {\it test functional} $\psi \colon \mathcal{G} \to \mathbb{R}$ and $k \in \mathbb{N}$, if the difference $\varphi - \psi$ attains its minimum on the set $\mathcal{G}_k$ (see \cref{G_k}) at some point $(t, w(\cdot)) \in \mathcal{G}_k \cap \mathcal{G}^\circ$, then
        \begin{displaymath}
            \boldsymbol{\partial}_t \psi(t, w(\cdot))
            + H \bigl( t, w(t), \boldsymbol{\nabla} \psi(t, w(\cdot)) \bigr)
            \leq 0.
        \end{displaymath}

    \noindent (b)
        For any $ci$-smooth test functional $\psi \colon \mathcal{G} \to \mathbb{R}$ and $k \in \mathbb{N}$, if the difference $\varphi - \psi$ attains its maximum on the set $\mathcal{G}_k$ at some point $(t, w(\cdot)) \in \mathcal{G}_k \cap \mathcal{G}^\circ$, then
         \begin{displaymath}
            \boldsymbol{\partial}_t \psi(t, w(\cdot))
            + H \bigl( t, w(t), \boldsymbol{\nabla} \psi(t, w(\cdot)) \bigr)
            \geq 0.
         \end{displaymath}

    \begin{theorem} \label{theorem_existence}
        Suppose that \cref{assumption_f,assumption_K,assumption_sigma_chi,assumption_K_2} hold.
        Then, the lower value functional $\rho_-$ is a viscosity solution of the Cauchy problem \cref{HJ,boundary_condition} with the {\rm lower Hamiltonian}
        \begin{equation} \label{lower_Hamiltonian}
            H_- (t, x, s)
            \doteq  \max_{v \in Q} \min_{u \in P}
            h(t, x, u, v, s)
            \quad \forall t \in [0, T], \ x, s \in \mathbb{R}^n
        \end{equation}
        and the boundary functional $\sigma$ from \cref{cost_functional}, while the upper value functional $\rho_+$ is a viscosity solution of the Cauchy problem \cref{HJ,boundary_condition} with the {\rm upper Hamiltonian}
        \begin{equation} \label{upper_Hamiltonian}
            H_+ (t, x, s)
            \doteq \min_{u \in P} \max_{v \in Q}
            h(t, x, u, v, s)
            \quad \forall t \in [0, T], \ x, s \in \mathbb{R}^n
        \end{equation}
        and the same boundary functional.
        In \cref{lower_Hamiltonian,upper_Hamiltonian}, the function $h$ is from \cref{h}.
    \end{theorem}

    Indeed, consider the lower value functional $\rho_-$ for definiteness.
    Recall that the inclusion $\rho_- \in \Phi$ is established in \cref{lemma_continuity}.
    Further, the fact that the functional $\rho_-$ meets the boundary condition \cref{boundary_condition} follows directly from the definition of this functional (see \cref{cost_functional_arbitrary,lower_upper_value_functionals}).
    In addition, taking into account that $\rho_-$ satisfies the dynamic programming principle (see \cref{proposition_DPP}), using \cref{proposition_ci_smooth} and the invariance property of the sets $\mathcal{G}_k$, $k \in \mathbb{N}$, with respect to system motions, and repeating the proof of \cite[Theorem 4.1]{Evans_Souganidis_1984} (see also, e.g., \cite[Chapter XI, Theorem 6.1]{Fleming_Soner_2006} and \cite[Theorem 3.3.6]{Yong_2015}), it can be verified that $\rho_-$ has both properties (a) and (b) with $H = H_-$.
    Thus, $\rho_-$ is indeed a viscosity solution of the Cauchy problem \cref{HJ,boundary_condition} with $H = H_-$ and the boundary functional $\sigma$ from \cref{cost_functional}.

    The next theorem presents a uniqueness result for viscosity solutions of the Cauchy problem \cref{HJ,boundary_condition}.
    \begin{theorem} \label{theorem_uniqueness}
        Let the kernel $K$ satisfy \cref{assumption_K,assumption_K_2}.
        Suppose that, in the Cauchy problem \cref{HJ,boundary_condition}, the Hamiltonian $H$ is continuous and there exists $c > 0$ such that
        \begin{displaymath}
            |H(t, x, s) - H(t, x, s^\prime)|
            \leq c (1 + \|x\|) \|s - s^\prime\|
            \quad \forall t \in [0, T], \ x, s, s^\prime \in \mathbb{R}^n.
        \end{displaymath}
        Then, this problem admits at most one viscosity solution.
    \end{theorem}

    In order to prove \cref{theorem_uniqueness}, consider the Lyapunov--Krasovskii functional
    \begin{equation} \label{nu}
        \begin{aligned}
            & \nu_\varepsilon \bigl( (t, w(\cdot)), (\tau, r(\cdot)) \bigr)
            \doteq \bigl( \varepsilon^{\frac{2}{q - 1}} + \|a(T; t, w(\cdot)) - a(T; \tau, r(\cdot))\|^2 \bigr)^{\frac{q}{2}} \\
            & \quad + \int_{0}^{T} \frac{\bigl( \varepsilon^{\frac{2}{q - 1}} + \|a(\xi; t, w(\cdot)) - a(\xi; \tau, r(\cdot))\|^2 \bigr)^{\frac{q}{2}}}
            {(T - \xi)^{(1 - \alpha - \alpha^\prime) q}} \, \rd \xi
            - C_1 \varepsilon^{\frac{q}{q - 1}}.
        \end{aligned}
    \end{equation}
    In the above, $(t, w(\cdot))$, $(\tau, r(\cdot)) \in \mathcal{G}$;
    $\varepsilon \in (0, 1]$ is a small parameter;
    $q \doteq 2 / (2 - \alpha)$,
    $\alpha^\prime \in (0, (1 - \alpha) \wedge (\alpha / 2))$,
    $C_1 \doteq 1 + T^{1 - (1 - \alpha - \alpha^\prime) q} / (1 - (1 - \alpha - \alpha^\prime) q)$;
    the functions $a(\cdot; t, w(\cdot))$, $a(\cdot; \tau, r(\cdot))$ are defined in accordance with \cref{a}.
    The necessary properties of the functional $\nu_\varepsilon$ are described in the lemma below.

    \begin{lemma} \label{lemma_nu_basic}
        The following statements hold.

        \noindent {\rm (a)}
            For every $\varepsilon \in (0, 1]$, the functional $\nu_\varepsilon$ is non-negative;
            the equalities
            \begin{displaymath}
                    \nu_\varepsilon \bigl( (t, w(\cdot)), (t, w(\cdot)) \bigr)
                    = 0,
                    \quad \nu_\varepsilon \bigl( (t, w(\cdot)), (\tau, r(\cdot)) \bigr)
                    = \nu_\varepsilon \bigl( (\tau, r(\cdot)), (t, w(\cdot)) \bigr),
            \end{displaymath}
            and
            \begin{displaymath}
                    \nu_\varepsilon \Bigl( \bigl( t^\prime, a_{t^\prime}(\cdot; t, w(\cdot)) \bigr), \bigl( \tau^\prime, a_{\tau^\prime}(\cdot; \tau, r(\cdot)) \bigr) \Bigr)
                    = \nu_\varepsilon \bigl( (t, w(\cdot)), (\tau, r(\cdot)) \bigr)
            \end{displaymath}
            are valid for all $(t, w(\cdot))$, $(\tau, r(\cdot)) \in \mathcal{G}$, $t^\prime \in [t, T]$, and $\tau^\prime \in [\tau, T]$; and the restriction of the functional $\nu_\varepsilon$ to the set $\mathcal{G}_k \times \mathcal{G}_k$ is (uniformly) continuous for every $k \in \mathbb{N}$.

        \noindent {\rm (b)}
            There exists $C_2 > 0$ such that, for any $\varepsilon \in (0, 1]$ and $(t, w(\cdot))$, $(\tau, r(\cdot)) \in \mathcal{G}$,
            \begin{displaymath}
                \begin{aligned}
                    & \|a(T; t, w(\cdot)) - a(T; \tau, r(\cdot))\|
                    + \int_{0}^{T} \frac{\|a(\xi; t, w(\cdot)) - a(\xi; \tau, r(\cdot))\|}{(T - \xi)^{1 - \alpha}} \, \rd \xi \\
                    & \quad \leq C_2 \Bigl( \nu_\varepsilon \bigl( (t, w(\cdot)), (\tau, r(\cdot)) \bigr) + C_1 \varepsilon^{\frac{q}{q - 1}} \Bigr)^{\frac{1}{q}}.
                \end{aligned}
            \end{displaymath}

        \noindent {\rm (c)}
            Let $k \in \mathbb{N}$ be fixed and let $(t^{(\varepsilon)}, w^{(\varepsilon)}(\cdot))$, $(\tau^{(\varepsilon)}, r^{(\varepsilon)}(\cdot)) \in \mathcal{G}_k$ be given for all $\varepsilon \in (0, 1]$.
            Suppose that $|t^{(\varepsilon)} - \tau^{(\varepsilon)}| \to 0$ and $\nu_\varepsilon ((t^{(\varepsilon)}, w^{(\varepsilon)}(\cdot)), (\tau^{(\varepsilon)}, r^{(\varepsilon)}(\cdot))) \to 0$ as $\varepsilon \to 0^+$.
            Then, we have $\dist((t^{(\varepsilon)}, w^{(\varepsilon)}(\cdot)), (\tau^{(\varepsilon)}, r^{(\varepsilon)}(\cdot))) \to 0$ as $\varepsilon \to 0^+$.

        \noindent {\rm (d)}
            Let $\varepsilon \in (0, 1]$ and $(\tau, r(\cdot)) \in \mathcal{G}$ be fixed.
            Then, the functional
            \begin{equation} \label{psi_varepsilon}
                \mu_\varepsilon^{(\tau, r(\cdot))} (t, w(\cdot))
                \doteq \nu_\varepsilon \bigl( (t, w(\cdot)), (\tau, r(\cdot)) \bigr)
                \quad \forall (t, w(\cdot)) \in \mathcal{G}
            \end{equation}
            is $ci$-smooth and its $ci$-derivatives are given by $\boldsymbol{\partial}_t \mu_\varepsilon^{(\tau, r(\cdot))} (t, w(\cdot)) = 0$ and
            \begin{equation} \label{lemma_nu_properties_item_c_derivative_x}
                \begin{aligned}
                    & \boldsymbol{\nabla} \mu_\varepsilon^{(\tau, r(\cdot))} (t, w(\cdot))
                    = \frac{q K(T, t)^\top \bigl( a(T; t, w(\cdot)) - a(T; \tau, r(\cdot)) \bigr)}
                    {\big( \varepsilon^{\frac{2}{q - 1}} + \|a(T; t, w(\cdot)) - a(T; \tau, r(\cdot))\|^2 \big)^{1 - \frac{q}{2}}} \\
                    & \quad + \int_{t}^{T} \frac{q K(\xi, t)^\top \bigl( a(\xi; t, w(\cdot)) - a(\xi; \tau, r(\cdot)) \bigr)}
                    {\big( \varepsilon^{\frac{2}{q - 1}} + \|a(\xi; t, w(\cdot)) - a(\xi; \tau, r(\cdot))\|^2 \big)^{1 - \frac{q}{2}} (T - \xi)^{(1 - \alpha - \alpha^\prime) q}} \, \rd \xi
                \end{aligned}
            \end{equation}
            for all $(t, w(\cdot)) \in \mathcal{G}^\circ$, where the superscript $^\top$ denotes transposition.

        \noindent {\rm (e)}
            For any $\theta \in (0, T)$ and $k \in \mathbb{N}$, there exists $C_3 > 0$ such that
            \begin{displaymath}
                \begin{aligned}
                    \| \boldsymbol{\nabla} \mu_\varepsilon^{(\tau, r(\cdot))} (t, w(\cdot)) \|
                    & \leq C_3 \Bigl( \nu_\varepsilon \bigl( (t, w(\cdot)), (\tau, r(\cdot)) \bigr) + C_1 \varepsilon^{\frac{q}{q - 1}} \Bigr)^{\frac{q - 1}{q}}, \\
                    \| \boldsymbol{\nabla} \mu_\varepsilon^{(\tau, r(\cdot))} (t, w(\cdot))
                    + \boldsymbol{\nabla} \mu_\varepsilon^{(t, w(\cdot))} (\tau, r(\cdot)) \|
                    & \leq C_3 |t - \tau|^{\alpha \wedge \beta}
                \end{aligned}
            \end{displaymath}
            for all $\varepsilon \in (0, 1]$ and $(t, w(\cdot))$, $(\tau, r(\cdot)) \in \mathcal{G}_k$ with $t \leq \theta$, $\tau \leq \theta$.
            Here, the number $\beta$ is taken from \cref{assumption_K}, {\rm (b)}.
    \end{lemma}

    Taking relation \cref{semigroup_property} and \cref{lemma_a_mapping} into account, property (a) can be proved similarly to \cite[Lemma 5.1]{Gomoyunov_2023_JDE}.
    Further, properties (b) and (c) can be verified by repeating the proofs of \cite[Lemmas 5.2 and 5.3]{Gomoyunov_2023_JDE}.
    Property (d) can be proved by essentially the same arguments as \cite[Lemma 5.4]{Gomoyunov_2023_JDE}.
    The only change that needs to be made is the following observation.
    Let $x(\cdot) \in \mathcal{X}(t, w(\cdot))$, denote $\ell(\cdot) \doteq \ell(\cdot; T, x(\cdot))$, and fix $\tau \in (t, T)$, $\xi \in (\tau, T]$.
    Then, due to \cref{a} and based on the integration by parts formula, we obtain
    \begin{displaymath}
        \begin{aligned}
            & a(\xi; \tau, x_\tau(\cdot)) - a(\xi; t, w(\cdot))
            = \int_{t}^{\tau} K(\xi, \eta) \ell(\eta) \, \rd \eta
            = \int_{t}^{\tau} \frac{K_\ast(\xi, \eta) \ell(\eta)}{(\xi - \eta)^{1 - \alpha}} \, \rd \eta \\
            & \quad = \frac{1}{(\xi - \tau)^{1 - \alpha}} \int_{t}^{\tau} K_\ast(\xi, \zeta) \ell(\zeta) \, \rd \zeta
            - (1 - \alpha) \int_{t}^{\tau} \int_{t}^{\eta} \frac{K_\ast(\xi, \zeta) \ell(\zeta)}{(\xi - \eta)^{2 - \alpha}} \, \rd \zeta \, \rd \eta
        \end{aligned}
    \end{displaymath}
    and, therefore,
    \begin{displaymath}
        \begin{aligned}
            & \biggl\| a(\xi; \tau, x_\tau(\cdot)) - a(\xi; t, w(\cdot))
            - K(\xi, t) \int_{t}^{\tau} \ell(\zeta) \, \rd \zeta \biggr\| \\
            & \quad \leq \biggl\| a(\xi; \tau, x_\tau(\cdot)) - a(\xi; t, w(\cdot))
            - \frac{K_\ast(\xi, t)}{(\xi - \tau)^{1 - \alpha}} \int_{t}^{\tau} \ell(\zeta) \, \rd \zeta \biggr\| \\
            & \qquad + \kappa_\ast \|\ell(\cdot)\|_{L_\infty([0, T], \mathbb{R}^n)} (\tau - t) \biggl( \frac{1}{(\xi - \tau)^{1 - \alpha}} - \frac{1}{(\xi - t)^{1 - \alpha}} \biggr) \\
            & \quad \leq \frac{\lambda \|\ell(\cdot)\|_{L_\infty([0, T], \mathbb{R}^n)} (\tau - t)^{\beta + 1}}{(\beta + 1)(\xi - \tau)^{1 - \alpha}} \\
            & \qquad + 2 \kappa_\ast \| \ell(\cdot)\|_{L_\infty([0, T], \mathbb{R}^n)} (\tau - t) \biggl( \frac{1}{(\xi - \tau)^{1 - \alpha}} - \frac{1}{(\xi - t)^{1 - \alpha}} \biggr),
        \end{aligned}
    \end{displaymath}
    where $\kappa_\ast$ and $\lambda$, $\beta$ are taken from \cref{mu_ast} and \cref{assumption_K}, (b).
    Finally, property (e) can be proved in the same way as \cite[Lemma 5.5]{Gomoyunov_2023_JDE}.
    Note that this is the place where \cref{assumption_K}, (b), is important.

    Thanks to properties (a)--(e) of the functional $\nu_\varepsilon$, \cref{theorem_uniqueness} can be proved by repeating the arguments from the proof of \cite[Theorem 5.1]{Gomoyunov_2023_JDE}, where, in view of appearance of the exponent $\alpha \wedge \beta$ instead of $\alpha$ in property (e), the basic functional $\Phi_\varepsilon \colon \mathcal{G} \times \mathcal{G} \to \mathbb{R}$ should be defined by
    \begin{displaymath}
        \begin{aligned}
            & \Phi_\varepsilon \bigl( (t, w(\cdot)), (\tau, r(\cdot)) \bigr) \\
            & \quad \doteq \varphi_1(t, w(\cdot)) - \varphi_2(\tau, r(\cdot))
            - (2 T - t - \tau) \zeta - \frac{(t - \tau)^2}{\varepsilon^{\frac{3}{\alpha \wedge \beta}}}
            - \frac{\nu_\varepsilon \bigl( (t, w(\cdot)), (\tau, r(\cdot)) \bigr)}{\varepsilon}
        \end{aligned}
    \end{displaymath}
    for all $(t, w(\cdot))$, $(\tau, r(\cdot)) \in \mathcal{G}$.
    Here, $\varphi_1$ and $\varphi_2$ are two viscosity solutions of the Cauchy problem \cref{HJ,boundary_condition} and $\zeta > 0$ is a suitably chosen number.

    Directly from \cref{theorem_existence,theorem_uniqueness}, we derive
    \begin{theorem} \label{theorem_characterization}
        Suppose that \cref{assumption_f,assumption_K,assumption_sigma_chi,assumption_Isaacs,assumption_K_2} hold.
        Then, the lower $\rho_-$ and upper $\rho_+$ value functionals coincide.
        Moreover, the {\rm value functional}
        \begin{equation} \label{value_functional}
            \rho(t, w(\cdot))
            \doteq \rho_-(t, w(\cdot))
            = \rho_+(t, w(\cdot))
            \quad \forall (t, w(\cdot)) \in \mathcal{G}
        \end{equation}
        of the game \cref{system,cost_functional} is a unique viscosity solution of the Cauchy problem \cref{HJ,boundary_condition} with the Hamiltonian
        \begin{equation} \label{Hamiltonian}
            H(t, x, s)
            \doteq H_- (t, x, s)
            = H_+ (t, x, s)
            \quad \forall t \in [0, T], \ x, s \in \mathbb{R}^n
        \end{equation}
        and the boundary functional $\sigma$ from \cref{cost_functional}.
    \end{theorem}

    Indeed, it suffices to note that $H_-$ and $H_+$ coincide according to \cref{assumption_Isaacs} and satisfy the conditions from \cref{theorem_uniqueness} by \cref{assumption_f,assumption_sigma_chi}.

    Thus, in view of equalities \cref{values_0}, we conclude by \cref{theorem_characterization} that the game \cref{system,cost_functional} has the value $\rho^0 = \rho(0, w(\cdot))$, where $w(0) \doteq y(0)$.

\section{Positional strategies}
\label{section_positional}

    Following \cite{Krasovskii_Subbotin_1988,Subbotin_1995} (see also, e.g. \cite{Lukoyanov_2000_PMM_Eng,Gomoyunov_2021_Mathematics}), as a first player's {\it positional strategy} in the considered game \cref{system,cost_functional}, we mean an arbitrary mapping $U \colon \mathcal{G}^\circ \to P$.
    Let $\Delta$ be a {\it partition} of the time interval $[0, T]$, i.e.,
    \begin{displaymath}
        \Delta
        \doteq \{\tau_j\}_{j = 1}^p,
        \quad \tau_1 = 0, \quad \tau_j < \tau_{j + 1} \quad \forall j \in \overline{1, p - 1}, \quad \tau_p = T,
    \end{displaymath}
    where $p \in \mathbb{N}$, $p \geq 2$.
    The pair $(U, \Delta)$ is called a {\it control law} of the first player.
    This law together with a second player's control $v(\cdot) \in \mathcal{V}[0, T]$ uniquely generates the first player's control $u(\cdot) \in \mathcal{U}[0, T]$ (and, respectively, the corresponding motion $x(\cdot)$ of system \cref{system}) by the following step-by-step rule:
    \begin{equation} \label{control_law_first_player}
        u(\tau)
        \doteq U(\tau_j, x_{\tau_j}(\cdot))
        \quad \forall \tau \in [\tau_j, \tau_{j + 1}), \ j \in \overline{1, p - 1}
    \end{equation}
    and, formally, $u(T) \doteq u_\ast$ for some fixed $u_\ast \in P$.
    In other words, at every time $\tau_j$, $j \in \overline{1, p - 1}$, the first player measures the history $x_{\tau_j}(\cdot)$ of the motion $x(\cdot)$ on $[0, \tau_j]$ (see \cref{x_t}), computes the value $u_j \doteq U(\tau_j, x_{\tau_j}(\cdot))$, and then applies the constant control $u(\tau) \doteq u_j$ until $\tau_{j + 1}$, when a new measurement of the history is taken.
    Denote the corresponding value of the cost functional \cref{cost_functional} by $J((U, \Delta), v(\cdot))$.

    The mapping that assigns to each second player's control $v(\cdot) \in \mathcal{V}[0, T]$ the first player's control $u(\cdot) \in \mathcal{U}[0, T]$ formed by $(U, \Delta)$ is a non-anticipative strategy of the first player by construction.
    Hence, in view of \cref{lower_value,game_value}, we have
    \begin{displaymath}
        \sup_{v(\cdot) \in \mathcal{V}[0, T]} J((U, \Delta), v(\cdot))
        \geq \rho^0.
    \end{displaymath}
    In this connection, given a number $\zeta > 0$, a positional strategy of the first player $U$ is called $\zeta${\it-optimal} if there exists $\delta > 0$ such that, for any partition $\Delta$ with the diameter $\diam(\Delta) \doteq \max \{\tau_{j + 1} - \tau_j \colon \, j \in \overline{1, p - 1} \} \leq \delta$,
    \begin{displaymath}
        \sup_{v(\cdot) \in \mathcal{V}[0, T]} J ((U, \Delta), v(\cdot))
        \leq \rho^0 + \zeta.
    \end{displaymath}

    Similarly, a positional strategy of the second player is a mapping $V \colon \mathcal{G}^\circ \to Q$.
    Given a partition $\Delta$ and a first player's control $u(\cdot) \in \mathcal{U}[0, T]$, let $J(u(\cdot), (V, \Delta))$ be the value of the cost functional \cref{cost_functional} that corresponds to the case where the second player's control $v(\cdot) \in \mathcal{V}[0, T]$ is formed by the control law $(V, \Delta)$ as follows:
    \begin{displaymath}
        v(\tau)
        \doteq V(\tau_j, x_{\tau_j}(\cdot))
        \quad \forall \tau \in [\tau_j, \tau_{j + 1}), \ j \in \overline{1, p - 1}
    \end{displaymath}
    and, formally, $v(T) \doteq v_\ast$ for some fixed $v_\ast \in Q$.
    Due to \cref{upper_value,game_value}, we get
    \begin{displaymath}
        \inf_{u(\cdot) \in \mathcal{U}[0, T]} J(u(\cdot), (V, \Delta))
        \leq \rho^0.
    \end{displaymath}
    Then, for a number $\zeta > 0$, a positional strategy of the second player $V$ is called $\zeta$-optimal if there exists $\delta > 0$ such that, for any partition $\Delta$ with $\diam(\Delta) \leq \delta$,
    \begin{displaymath}
        \inf_{u(\cdot) \in \mathcal{U}[0, T]} J(u(\cdot), (V, \Delta))
        \geq \rho^0 - \zeta.
    \end{displaymath}

\section{Construction of optimal positional strategies}
\label{section_optimal_positional}

    Let us take the compact sets $\mathcal{X}_1$ and $\mathcal{G}_1$ (see \cref{X_k,G_k}) and, for every $t \in [0, T]$, denote by $\mathcal{G}_1(t)$ the set of functions $w(\cdot)$ such that $(t, w(\cdot)) \in \mathcal{G}_1$.

    For any $(t, w(\cdot)) \in \mathcal{G}$ and $\varepsilon \in (0, 1]$, recalling that the restriction of the value functional $\rho$ (see \cref{value_functional}) to the set $\mathcal{G}_1$ is continuous owing to \cref{lemma_continuity} and that the restriction of the Lyapunov--Krasovskii functional $\nu_\varepsilon$ (see \cref{nu}) to the set $\mathcal{G}_1 \times \mathcal{G}_1$ is continuous by \cref{lemma_nu_basic}, (a), we can introduce the values
    \begin{equation} \label{rho_varepsilon}
        \begin{aligned}
            \rho_\varepsilon^- (t, w(\cdot))
            & \doteq \min_{r(\cdot) \in \mathcal{G}_1(t)} \biggl( \rho(t, r(\cdot)) + \frac{\nu_\varepsilon \bigl( (t, w(\cdot)), (t, r(\cdot)) \bigr)}{\varepsilon} \biggr), \\
            \rho_\varepsilon^+ (t, w(\cdot))
            & \doteq \max_{r(\cdot) \in \mathcal{G}_1(t)} \biggl( \rho(t, r(\cdot)) - \frac{\nu_\varepsilon \bigl( (t, w(\cdot)), (t, r(\cdot)) \bigr)}{\varepsilon} \biggr)
        \end{aligned}
    \end{equation}
    and then choose functions $r_\varepsilon^-(\cdot; t, w(\cdot))$ and $r_\varepsilon^+(\cdot; t, w(\cdot))$ at which, respectively, the minimum and maximum in \cref{rho_varepsilon} are attained.

    For every $\varepsilon \in (0, 1]$, let us consider players' positional strategies $U_\varepsilon^0 \colon \mathcal{G}^\circ \to P$ and $V_\varepsilon^0 \colon \mathcal{G}^\circ \to Q$ such that
    \begin{equation} \label{U_circ}
        \begin{aligned}
            U_\varepsilon^0 (t, w(\cdot))
            & \in \argmin{u \in P} \max_{v \in Q}
            h \biggl(t, w(t), u, v, \frac{\boldsymbol{\nabla} \mu_\varepsilon^{(t, r_\varepsilon^-(\cdot; t, w(\cdot)))} (t, w(\cdot))}{\varepsilon} \biggr), \\
            V_\varepsilon^0 (t, w(\cdot))
            & \in \argmax{v \in Q} \min_{u \in P}
            h \biggl(t, w(t), u, v, - \frac{\boldsymbol{\nabla} \mu_\varepsilon^{(t, r_\varepsilon^+(\cdot; t, w(\cdot)))} (t, w(\cdot)) }{\varepsilon} \biggr)
        \end{aligned}
    \end{equation}
    for all $(t, w(\cdot)) \in \mathcal{G}^\circ$, where $h$ and $\boldsymbol{\nabla} \mu_\varepsilon$ are defined according to \cref{h,lemma_nu_properties_item_c_derivative_x}.

    \begin{theorem} \label{theorem_positional}
        Suppose that \cref{assumption_f,assumption_K,assumption_sigma_chi,assumption_Isaacs,assumption_K_2} hold.
        Then, for any $\zeta > 0$, there exists $\varepsilon_\ast \in (0, 1]$ such that, for every $\varepsilon \in (0, \varepsilon_\ast]$, the players' positional strategies $U_\varepsilon^0$ and $V_\varepsilon^0$ are $\zeta$-optimal.
    \end{theorem}

    Before proceeding with the proof of \cref{theorem_positional}, we establish some additional properties of the Lyapunov--Krasovskii functional $\nu_\varepsilon$ and the value functional $\rho$.
    \begin{lemma} \label{lemma_nu_advanced}
        The following statements hold.

        \noindent {\rm (a)}
            For any $\varepsilon \in (0, 1]$, the mapping below is continuous:
            \begin{equation} \label{mapping_nu_6}
                [0, T) \times \mathcal{X}_1 \times [0, T] \times \mathcal{X}_1
                \ni (\tau, x(\cdot), \tau^\prime, x^\prime(\cdot))
                \mapsto \boldsymbol{\nabla} \mu_\varepsilon^{(\tau^\prime, x^\prime_{\tau^\prime}(\cdot))} (\tau, x_\tau(\cdot))
                \in \mathbb{R}^n.
            \end{equation}

        \noindent {\rm (b)}
            Suppose that $\varepsilon \in (0, 1]$, $\theta \in (0, T)$, and $x(\cdot)$, $x^\prime(\cdot) \in \mathcal{X}_1$ and consider the function $\omega(\tau) \doteq \nu_\varepsilon((\tau, x_\tau(\cdot)), (\tau, x^\prime_\tau(\cdot)))$, $\tau \in [0, \theta]$.
            Then, $\omega(\cdot)$ is Lipschitz continuous and
            \begin{equation} \label{nu.7_main}
                \dot{\omega}(\tau)
                = \langle \boldsymbol{\nabla} \mu_\varepsilon^{(\tau, x^\prime_\tau(\cdot))} (\tau, x_\tau(\cdot)), \ell(\tau; T, x(\cdot)) - \ell(\tau; T, x^\prime(\cdot)) \rangle
                \quad \text{for a.e. } \tau \in [0, \theta],
            \end{equation}
            where $\ell(\cdot; T, x(\cdot))$ and $\ell(\cdot; T, x^\prime(\cdot))$ are determined by $x(\cdot)$ and $x^\prime(\cdot)$ according to \cref{representation_via_l_t_w}.

        \noindent {\rm (c)}
            For any $R > 0$ and $\varkappa > 0$, there exists $\varepsilon_\ast \in (0, 1]$ such that, for any $\varepsilon \in (0, \varepsilon_\ast]$, $t \in [0, T]$, and $w(\cdot)$, $w^\prime(\cdot) \in \mathcal{G}_1(t)$ satisfying the condition $\nu_\varepsilon((t, w(\cdot)), (t, w^\prime(\cdot))) \leq R \varepsilon$, the inequality $\|w(\cdot) - w^\prime(\cdot)\|_{C([0, t], \mathbb{R}^n)} \leq \varkappa$ is valid.
    \end{lemma}
    \begin{proof}
        Let us verify property (a).
        Consider the function
        \begin{displaymath}
            g(y)
            \doteq \frac{y}{(\varepsilon^{\frac{2}{q - 1}} + \|y\|^2)^{1 - \frac{q}{2}}}
            \quad \forall y \in \mathbb{R}^n
        \end{displaymath}
        and observe that it satisfies the estimate
        \begin{displaymath}
            \|g(y) - g(y^\prime)\|
            \leq \frac{3 - q}{\varepsilon^{\frac{2 - q}{q - 1}}} \|y - y^\prime\|
            \quad \forall y, y^\prime \in \mathbb{R}^n.
        \end{displaymath}
        Hence, in view of \cref{lemma_nu_properties_item_c_derivative_x}, for any $(t, w(\cdot)) \in \mathcal{G}^\circ$ and $(\tau, r(\cdot))$, $(\tau^\prime, r^\prime(\cdot)) \in \mathcal{G}$,
        taking $\kappa_\ast$ from \cref{mu_ast} and noting that (see, e.g., \cite[equality (A.6)]{Gomoyunov_2023_JDE})
        \begin{displaymath}
            \begin{aligned}
                \int_{t}^{T} \frac{\rd \xi}{(\xi - t)^{1 - \alpha} (T - \xi)^{(1 - \alpha - \alpha^\prime) q}}
                & = \mathrm{B}(\alpha, 1 - (1 - \alpha - \alpha^\prime) q) (T - t)^{\alpha - (1 - \alpha - \alpha^\prime) q} \\
                & \leq \frac{\mathrm{B}(\alpha, 1 - (1 - \alpha - \alpha^\prime) q) T^{1 - (1 - \alpha - \alpha^\prime) q}}{(T - t)^{1 - \alpha}},
            \end{aligned}
        \end{displaymath}
        where $\mathrm{B}$ is the Euler beta-function, we derive
        \begin{displaymath}
            \begin{aligned}
                & \| \boldsymbol{\nabla} \mu_\varepsilon^{(\tau, r(\cdot))} (t, w(\cdot))
                - \boldsymbol{\nabla} \mu_\varepsilon^{(\tau^\prime, r^\prime(\cdot))} (t, w(\cdot)) \| \\
                & \quad \leq \frac{q \kappa_\ast}{(T - t)^{1 - \alpha}} \bigl\| g\bigl( a(T; t, w(\cdot)) - a(T; \tau, r(\cdot)) \bigr) - g\bigl( a(T; t, w(\cdot)) - a(T; \tau^\prime, r^\prime(\cdot)) \bigr) \bigr\| \\
                & \qquad + q \kappa_\ast \int_{t}^{T} \frac{\bigl\| g \bigl( a(\xi; t, w(\cdot)) - a(\xi; \tau, r(\cdot)) \bigr) - g \bigl( a(\xi; t, w(\cdot)) - a(\xi; \tau^\prime, r^\prime(\cdot)) \bigr) \bigr\|}{(\xi - t)^{1 - \alpha} (T - \xi)^{(1 - \alpha - \alpha^\prime) q}}  \, \rd \xi \\
                & \quad \leq \frac{\lambda}{(T - t)^{1 - \alpha}} \| a(\cdot; \tau, r(\cdot)) - a(\cdot; \tau^\prime, r^\prime(\cdot)) \|_{C([0, T], \mathbb{R}^n)}
            \end{aligned}
        \end{displaymath}
        with
        \begin{displaymath}
            \lambda
            \doteq \frac{(3 - q) q \kappa_\ast}{\varepsilon^{\frac{2 - q}{q - 1}}}
            \Bigl( 1 + \mathrm{B}(\alpha, 1 - (1 - \alpha - \alpha^\prime) q) T^{1 - (1 - \alpha - \alpha^\prime) q} \Bigr).
        \end{displaymath}
        Due to \cref{lemma_a_mapping}, the obtained estimate implies that, for every $\theta \in (0, T)$, the family of mappings $\mathcal{G}_1 \ni (\tau, r(\cdot)) \mapsto \boldsymbol{\nabla} \mu_\varepsilon^{(\tau, r(\cdot))} (t, w(\cdot)) \in \mathbb{R}^n$ parameterized by $(t, w(\cdot)) \in \mathcal{G}^\circ$ with $t \leq \theta$ is uniformly equicontinuous.
        Together with continuity of the restriction of $\boldsymbol{\nabla} \mu_\varepsilon^{(\tau, r(\cdot))}$ to $\mathcal{G}_1 \cap \mathcal{G}^\circ$ for every $(\tau, r(\cdot)) \in \mathcal{G}$ (see \cref{lemma_nu_basic}, (d)), this yields continuity of the mapping $(\mathcal{G}_1 \cap \mathcal{G}^\circ) \times \mathcal{G}_1 \ni ((t, w(\cdot)), (\tau, r(\cdot))) \mapsto \boldsymbol{\nabla} \mu_\varepsilon^{(\tau, r(\cdot))} (t, w(\cdot)) \in \mathbb{R}^n$.
        Thus, recalling that mapping \cref{mapping_basic} is continuous, we get continuity of mapping \cref{mapping_nu_6}.

        Further, let us prove property (b).
        By property (a), there exists $M > 0$ such that $\| \boldsymbol{\nabla} \mu_\varepsilon^{(\tau^\prime, x^\prime_{\tau^\prime}(\cdot))} (\tau, x_\tau(\cdot))\| \leq M$ and $\| \boldsymbol{\nabla} \mu_\varepsilon^{(\tau, x_\tau(\cdot))} (\tau^\prime, x^\prime_{\tau^\prime}(\cdot))\| \leq M$ for all $\tau$, $\tau^\prime \in [0, \theta]$.
        Fix $\tau$, $\tau^\prime \in [0, \theta]$ with $\tau^\prime \geq \tau$.
        Then, in view of \cref{lemma_nu_basic}, (a), and \cref{psi_varepsilon}, we have
        \begin{displaymath}
            \begin{aligned}
                \omega(\tau^\prime) - \omega(\tau)
                & = \mu_\varepsilon^{(\tau^\prime, x^\prime_{\tau^\prime}(\cdot))} (\tau^\prime, x_{\tau^\prime}(\cdot))
                - \mu_\varepsilon^{(\tau^\prime, x^\prime_{\tau^\prime}(\cdot))} (\tau, x_\tau(\cdot)) \\
                & \quad + \mu_\varepsilon^{(\tau, x_\tau(\cdot))} (\tau^\prime, x^\prime_{\tau^\prime}(\cdot))
                - \mu_\varepsilon^{(\tau, x_\tau(\cdot))} (\tau, x^\prime_\tau(\cdot)).
            \end{aligned}
        \end{displaymath}
        Therefore, based on \cref{lemma_nu_basic}, (d), and applying \cref{proposition_ci_smooth}, we derive
        \begin{equation} \label{omega_difference}
            \begin{aligned}
                \omega(\tau^\prime) - \omega(\tau)
                & = \int_{\tau}^{\tau^\prime} \Bigl( \langle \boldsymbol{\nabla} \mu_\varepsilon^{(\tau^\prime, x^\prime_{\tau^\prime}(\cdot))} (\xi, x_\xi(\cdot)), \ell(\xi; T, x(\cdot)) \rangle \\
                & \quad + \langle \boldsymbol{\nabla} \mu_\varepsilon^{(\tau, x_\tau(\cdot))} (\xi, x^\prime_\xi(\cdot)), \ell(\xi; T, x^\prime(\cdot)) \rangle \Bigr) \, \rd \xi.
            \end{aligned}
        \end{equation}
        Hence,
        \begin{displaymath}
            |\omega(\tau^\prime) - \omega(\tau)|
            \leq M \bigl( \|\ell(\cdot; T, x(\cdot))\|_{L_\infty([0, T], \mathbb{R}^n)}
            + \|\ell(\cdot; T, x^\prime(\cdot))\|_{L_\infty([0, T], \mathbb{R}^n)} \bigr)
            (\tau^\prime - \tau),
        \end{displaymath}
        which implies that the function $\omega(\cdot)$ is Lipschitz continuous.

        Now, let $\mathcal{T}$ be the set of $\tau \in (0, \theta)$ such that the derivative $\dot{\omega}(\tau)$ exists and $\tau$ is a Lebesgue point of both functions $\ell(\cdot; T, x(\cdot))$ and $\ell(\cdot; T, x^\prime(\cdot))$.
        For every $\tau \in \mathcal{T}$, due to \cref{omega_difference} and using property (a), we obtain
        \begin{displaymath}
            \begin{aligned}
                \dot{\omega}(\tau)
                & = \lim_{\tau^\prime \to \tau^+} \frac{\omega(\tau^\prime) - \omega(\tau)}{\tau^\prime - \tau} \\
                & = \biggl\langle \boldsymbol{\nabla} \mu_\varepsilon^{(\tau, x^\prime_\tau(\cdot))} (\tau, x_\tau(\cdot)), \lim_{\tau^\prime \to \tau^+} \frac{1}{\tau^\prime - \tau}
                \int_{\tau}^{\tau^\prime} \ell(\xi; T, x(\cdot)) \, \rd \xi \biggr\rangle \\
                & \quad + \biggl\langle \boldsymbol{\nabla} \mu_\varepsilon^{(\tau, x_\tau(\cdot))} (\tau, x^\prime_\tau(\cdot)), \lim_{\tau^\prime \to \tau^+} \frac{1}{\tau^\prime - \tau}
                \int_{\tau}^{\tau^\prime} \ell(\xi; T, x^\prime(\cdot)) \, \rd \xi \biggr\rangle \\
                & = \langle \boldsymbol{\nabla} \mu_\varepsilon^{(\tau, x^\prime_\tau(\cdot))} (\tau, x_\tau(\cdot)), \ell(\tau; T, x(\cdot)) \rangle
                + \langle \boldsymbol{\nabla} \mu_\varepsilon^{(\tau, x_\tau(\cdot))} (\tau, x^\prime_\tau(\cdot)), \ell(\tau; T, x^\prime(\cdot)) \rangle.
            \end{aligned}
        \end{displaymath}
        As a result, and since $\boldsymbol{\nabla} \mu_\varepsilon^{(\tau, x_\tau(\cdot))} (\tau, x^\prime_\tau(\cdot)) = - \boldsymbol{\nabla} \mu_\varepsilon^{(\tau, x^\prime_\tau(\cdot))} (\tau, x_\tau(\cdot))$ by \cref{lemma_nu_properties_item_c_derivative_x} and the set $[0, \theta] \setminus \mathcal{T}$ has zero Lebesgue measure, we come to \cref{nu.7_main}.

        Finally, observe that property (c) follows directly from \cref{lemma_nu_basic}, (c).
    \end{proof}

    \begin{lemma} \label{lemma_stability}
        The following statements hold.

        \noindent {\rm (a)}
            For any $(t, w(\cdot)) \in \mathcal{G}$, $\theta \in [t, T]$, $v \in Q$, and $\eta > 0$, there exists $u(\cdot) \in \mathcal{U}[t, T]$ such that, for the system motion $x(\cdot) \doteq x(\cdot; t, w(\cdot), u(\cdot), v(\cdot))$ with $v(\tau) \doteq v$, $\tau \in [t, T]$,
            \begin{displaymath}
                \rho(\theta, x_\theta(\cdot)) + \int_{t}^{\theta} \chi(\tau, x(\tau), u(\tau), v) \, \rd \tau
                \leq \rho(t, w(\cdot)) + \eta.
            \end{displaymath}

        \noindent {\rm (b)}
            For any $(t, w(\cdot)) \in \mathcal{G}$, $\theta \in [t, T]$, $u \in P$, and $\eta > 0$, there exists $v(\cdot) \in \mathcal{V}[t, T]$ such that, for the system motion $x(\cdot) \doteq x(\cdot; t, w(\cdot), u(\cdot), v(\cdot))$ with $u(\tau) \doteq u$, $\tau \in [t, T]$,
            \begin{displaymath}
                \rho(\theta, x_\theta(\cdot)) + \int_{t}^{\theta} \chi(\tau, x(\tau), u, v(\tau)) \, \rd \tau
                \geq \rho(t, w(\cdot)) - \eta.
            \end{displaymath}
    \end{lemma}

    In the positional differential games theory \cite{Krasovskii_Subbotin_1988,Subbotin_1995}, properties (a) and (b) are called $u$- and $v$-stability properties of the value functional $\rho$ respectively.
    These properties can be derived from \cref{value_functional,proposition_DPP} in a direct way.

    \begin{proof}[Proof of \cref{theorem_positional}]
        The proof mainly follows the scheme from, e.g., \cite{Garnysheva_Subbotin_1994_Eng} and \cite[Theorem 12.3]{Subbotin_1995} (see also, e.g., \cite[Theorem 1]{Lukoyanov_2004_PMM_Eng} and \cite[Theorem 1]{Gomoyunov_2023_Motor}).

        Fix $\zeta > 0$.
        Since the function $\chi$ is continuous by \cref{assumption_sigma_chi}, (a), there exists $M > 0$ such that $|\chi(\tau, x(\tau), u, v)| \leq M$ for all $\tau \in [0, T]$, $x(\cdot) \in \mathcal{X}_1$, $u \in P$, and $v \in Q$.
        By \cref{lemma_continuity}, the value functional $\rho$ belongs to $\Phi$.
        Hence, first, the restriction of $\rho$ to $\mathcal{G}_1$ is continuous, and, consequently, there exists $R > 0$ such that $2 |\rho(t, w(\cdot))| \leq R$, $(t, w(\cdot)) \in \mathcal{G}_1$; there exists $\varkappa_\ast > 0$ such that $|\rho(t, w(\cdot)) - \rho(t, w^\prime(\cdot))| \leq \zeta / 4$ for all $t \in [0, T]$ and $w(\cdot)$, $w^\prime(\cdot) \in \mathcal{G}_1(t)$ satisfying the condition $\|w(\cdot) - w^\prime(\cdot)\|_{C([0, t], \mathbb{R}^n)} \leq \varkappa_\ast$; and there exists $\theta \in (0, T)$ such that $|\rho(T, x(\cdot)) - \rho(\tau, x_\tau(\cdot))| + M (T - \tau) \leq \zeta / 4$ for all $\tau \in [\theta, T]$ and $x(\cdot) \in \mathcal{X}_1$.
        Second, we can find $\lambda > 0$ such that inequality \cref{varphi.2_inequality} with $\varphi$ replaced by $\rho$ is valid for all $t \in [0, T]$ and $w(\cdot)$, $w^\prime(\cdot) \in \mathcal{G}_1(t)$.
        Further, using \cref{assumption_f}, (b), and \cref{assumption_sigma_chi}, (b), and recalling definition \cref{Hamiltonian} of $H$, let us take $\lambda_H > 0$ such that, for any $t \in [0, T]$, $w(\cdot)$, $w^\prime(\cdot) \in \mathcal{G}_1(t)$, and $s \in \mathbb{R}^n$,
        \begin{equation} \label{lambda_H}
            |H(t, w(t), s) - H(t, w^\prime(t), s)|
            \leq \lambda_H (1 + \|s\|) \|w(t) - w^\prime(t)\|.
        \end{equation}
        Put $\varkappa \doteq \min\{ \varkappa_\ast, \zeta / (8 T \lambda_H (1 + C_3 (\lambda C_2 + C_1^{\frac{q - 1}{q}})))\}$, where $C_1$, $q$ are from \cref{nu}, $C_2$ is from \cref{lemma_nu_basic}, (b), and $C_3$ is from \cref{lemma_nu_basic}, (e), for $(T + \theta) / 2$ and $k = 1$.
        By $R$ and $\varkappa$, let us choose $\varepsilon_\ast \in (0, 1]$ based on \cref{lemma_nu_advanced}, (c).

        Fix $\varepsilon \in (0, \varepsilon_\ast]$.
        Let us show that the first player's positional control strategy $U_\varepsilon^0$ is $\zeta$-optimal; the proof for $V_\varepsilon^0$ is similar.
        In view of \cref{lemma_nu_advanced}, (a), and continuity of $h$ (see \cref{h}), there exists $\delta \in (0, (T - \theta) / 2]$ such that, for any $x(\cdot)$, $x^\prime(\cdot)$, $x^{\prime \prime}(\cdot) \in \mathcal{X}_1$, $u \in P$, $v \in Q$, and $\tau$, $\tau^\prime \in [0, (T + \theta) / 2]$ with $|\tau - \tau^\prime| \leq \delta$,
        \begin{displaymath}
            \begin{aligned}
                & \biggl|
                h \biggl( \tau, x^{\prime \prime}(\tau), u, v, \frac{\boldsymbol{\nabla} \mu_\varepsilon^{(\tau, x^\prime_\tau(\cdot))} (\tau, x_\tau(\cdot))}{\varepsilon} \biggr) \\
                & \quad - h \biggl( \tau^\prime, x^{\prime \prime}(\tau^\prime), u, v, \frac{\boldsymbol{\nabla} \mu_\varepsilon^{(\tau^\prime, x^\prime_{\tau^\prime}(\cdot))} (\tau^\prime, x_{\tau^\prime}(\cdot))}{\varepsilon} \biggr) \biggr|
                \leq \frac{\zeta}{16 T}.
            \end{aligned}
        \end{displaymath}
        Take a partition $\Delta$ with $\diam(\Delta) \leq \delta$ and consider a motion $x(\cdot)$ of system \cref{system} generated by the first player's control law $(U_\varepsilon^0, \Delta)$ together with a second player's control $v(\cdot) \in \mathcal{V}[0, T]$.
        Denote by $u(\cdot) \in \mathcal{U}[0, T]$ the corresponding control of the first player.
        According to the definition of $\zeta$-optimality, we need to verify that
        \begin{equation} \label{lemma_general_case_u_main}
            J(u(\cdot), v(\cdot))
            \leq \rho^0 + \zeta.
        \end{equation}

        Choose $m \in \overline{2, p}$ from the condition $\tau_{m - 1} < \theta \leq \tau_m$.
        Note that $\tau_m \leq (T + \theta) / 2$ since $\delta \leq (T - \theta) / 2$.
        Using the fact that the value functional $\rho$ satisfies the boundary condition \cref{boundary_condition} and the choice of $M$ and $\theta$, we derive
        \begin{equation} \label{J_first}
            \begin{aligned}
                J(u(\cdot), v(\cdot))
                & \leq \rho(T, x(\cdot)) + \int_{0}^{\tau_m} \chi(\tau, x(\tau), u(\tau), v(\tau)) \, \rd \tau
                + M (T - \tau_m) \\
                & \leq \rho(\tau_m, x_{\tau_m}(\cdot))
                + \int_{0}^{\tau_m} \chi(\tau, x(\tau), u(\tau), v(\tau)) \, \rd \tau
                + \frac{\zeta}{4}.
            \end{aligned}
        \end{equation}

        Denote $r^{[j]}(\cdot) \doteq r_\varepsilon^-(\cdot; \tau_j, x_{\tau_j}(\cdot))$, $j \in \overline{1, m}$.
        By \cref{rho_varepsilon} and \cref{lemma_nu_basic}, (a), we have
        \begin{displaymath}
            \rho(\tau_j, r^{[j]}(\cdot)) + \frac{\nu_\varepsilon\bigl( (\tau_j, x_{\tau_j}(\cdot)), (\tau_j, r^{[j]}(\cdot)) \bigr)}{\varepsilon}
            = \rho_\varepsilon^-(\tau_j, x_{\tau_j}(\cdot))
            \leq \rho(\tau_j, x_{\tau_j}(\cdot))
        \end{displaymath}
        and, therefore,
        \begin{equation} \label{j_basic}
            \frac{\nu_\varepsilon\bigl( (\tau_j, x_{\tau_j}(\cdot)), (\tau_j, r^{[j]}(\cdot)) \bigr)}{\varepsilon}
            \leq \rho(\tau_j, x_{\tau_j}(\cdot)) - \rho(\tau_j, r^{[j]}(\cdot)).
        \end{equation}
        Due to the choice of $R$, this inequality implies that $\nu_\varepsilon((\tau_j, x_{\tau_j}(\cdot)), (\tau_j, r^{[j]}(\cdot))) \leq R \varepsilon$, and, hence, by the choice of $\varepsilon_\ast$, we get
        \begin{equation} \label{x-z_r}
            \|x_{\tau_j}(\cdot) - r^{[j]}(\cdot)\|_{C([0, \tau_j], \mathbb{R}^n)}
            \leq \varkappa.
        \end{equation}
        In addition, according to the choice of $\lambda$ and $C_2$, we obtain
        \begin{displaymath}
            \rho(\tau_j, x_{\tau_j}(\cdot)) - \rho(\tau_j, r^{[j]}(\cdot))
            \leq \lambda C_2 \Bigl( \nu_\varepsilon \bigl( (\tau_j, x_{\tau_j}(\cdot)), (\tau_j, r^{[j]}(\cdot)) \bigr) + C_1 \varepsilon^{\frac{q}{q - 1}} \Bigr)^{\frac{1}{q}}.
        \end{displaymath}
        Then, taking into account that
        \begin{displaymath}
            \frac{C_1 \varepsilon^{\frac{q}{q - 1}}}{\varepsilon}
            = C_1^{\frac{q - 1}{q}} \bigl( C_1 \varepsilon^{\frac{q}{q - 1}} \bigr)^{\frac{1}{q}}
            \leq C_1^{\frac{q - 1}{q}} \Bigl( \nu_\varepsilon \bigl( (\tau_j, x_{\tau_j}(\cdot)), (\tau_j, r^{[j]}(\cdot)) \bigr)
            + C_1 \varepsilon^{\frac{q}{q - 1}} \Bigr)^{\frac{1}{q}},
        \end{displaymath}
        it follows from \cref{j_basic} that
        \begin{displaymath}
            \frac{\Bigl( \nu_\varepsilon \bigl( (\tau_j, x_{\tau_j}(\cdot)), (\tau_j, r^{[j]}(\cdot)) \bigr)
            + C_1 \varepsilon^{\frac{q}{q - 1}} \Bigr)^{\frac{q - 1}{q}}}{\varepsilon}
            \leq \lambda C_2 + C_1^{\frac{q - 1}{q}}.
        \end{displaymath}
        Recalling the choice of $C_3$, we derive $\|\boldsymbol{\nabla} \mu_\varepsilon^{(\tau_j, r^{[j]}(\cdot))} (\tau_j, x_{\tau_j}(\cdot))\| / \varepsilon \leq C_3 (\lambda C_2 + C_1^{\frac{q - 1}{q}})$.
        Consequently, by \cref{lambda_H}, \cref{x-z_r}, and the choice of $\varkappa$, the inequality below holds:
        \begin{equation} \label{H-H}
            \begin{aligned}
                & \biggl| H \biggl( \tau_j, x(\tau_j), \frac{\boldsymbol{\nabla} \mu_\varepsilon^{(\tau_j, r^{[j]}(\cdot))} (\tau_j, x_{\tau_j}(\cdot))}{\varepsilon} \biggr) \\
                & \quad - H \biggl( \tau, r^{[j]}(\tau_j), \frac{\boldsymbol{\nabla} \mu_\varepsilon^{(\tau_j, r^{[j]}(\cdot))} (\tau_j, x_{\tau_j}(\cdot))}{\varepsilon} \biggr) \biggr|
                \leq \frac{\zeta}{8 T}.
            \end{aligned}
        \end{equation}

        According to \cref{x-z_r} with $j \doteq m$, the inequality $\varkappa \leq \varkappa_\ast$, and \cref{lemma_nu_basic}, (a), we have
        $\rho(\tau_m, x_{\tau_m}(\cdot)) \leq \rho(\tau_m, r^{[m]}(\cdot)) + \zeta / 4 \leq \rho_\varepsilon^-(\tau_m, x_{\tau_m}(\cdot)) + \zeta / 4$.
        Therefore, observing that $\rho^0 = \rho(\tau_1, x_{\tau_1}(\cdot)) \geq \rho_\varepsilon^-(\tau_1, x_{\tau_1}(\cdot))$ and using \cref{J_first}, we conclude that, in order to verify \cref{lemma_general_case_u_main}, it remains to show that
        \begin{displaymath}
            \rho_\varepsilon^-(\tau_m, x_{\tau_m}(\cdot))
            + \int_{0}^{\tau_m} \chi(\tau, x(\tau), u(\tau), v(\tau)) \, \rd \tau
            - \rho_\varepsilon^-(\tau_1, x_{\tau_1}(\cdot))
            \leq \frac{\zeta}{2}.
        \end{displaymath}
        In its turn, to this end, it suffices to fix $j \in \overline{1, m - 1}$ and prove that
        \begin{equation} \label{inequality_step_main}
            \begin{aligned}
                & \rho_\varepsilon^-(\tau_{j + 1}, x_{\tau_{j + 1}}(\cdot))
                + \int_{\tau_j}^{\tau_{j + 1}} \chi(\tau, x(\tau), u(\tau), v(\tau)) \, \rd \tau
                - \rho_\varepsilon^-(\tau_j, x_{\tau_j}(\cdot)) \\
                & \quad \leq \frac{\zeta (\tau_{j + 1} - \tau_j)}{2 T}.
            \end{aligned}
        \end{equation}

        Let us choose
        \begin{displaymath}
            v_j \in \argmax{v \in Q} \min_{u \in P} h \biggl( \tau_j, r^{[j]}(\tau_j), u, v, \frac{\boldsymbol{\nabla} \mu_\varepsilon^{(\tau_j, r^{[j]}(\cdot))} (\tau_j, x_{\tau_j}(\cdot))}{\varepsilon} \biggr).
        \end{displaymath}
        Applying \cref{lemma_stability}, (a), we obtain that there exists $u^{[j]}(\cdot) \in \mathcal{U}[\tau_j, T]$ such that
        \begin{displaymath}
            \rho(\tau_{j + 1}, x^{[j]}_{\tau_{j + 1}}(\cdot))
            + \int_{\tau_j}^{\tau_{j + 1}} \chi(\tau, x^{[j]}(\tau), u^{[j]}(\tau), v_j) \, \rd \tau
            \leq \rho(\tau_j, r^{[j]}(\cdot)) + \frac{\zeta (\tau_{j + 1} - \tau_j)}{4 T},
        \end{displaymath}
        where $x^{[j]}(\cdot) \doteq x(\cdot; \tau_j, r^{[j]}(\cdot), u^{[j]}(\cdot), v^{[j]}(\cdot))$ and $v^{[j]}(\tau) \doteq v_j$, $\tau \in [\tau_j, T]$.
        Taking \cref{rho_varepsilon} into account and noting that $x^{[j]}_{\tau_{j + 1}}(\cdot) \in \mathcal{G}_1(\tau_{j + 1})$, we derive
        \begin{align*}
            & \rho_\varepsilon^-(\tau_{j + 1}, x_{\tau_{j + 1}}(\cdot))
            \leq \rho(\tau_{j + 1}, x^{[j]}_{\tau_{j + 1}}(\cdot))
            + \frac{\nu_\varepsilon \bigl( (\tau_{j + 1}, x_{\tau_{j + 1}}(\cdot)), (\tau_{j + 1}, x^{[j]}_{\tau_{j + 1}}(\cdot)) \bigr)}{\varepsilon} \\
            & \quad \leq \rho(\tau_j, r^{[j]}(\cdot))
            - \int_{\tau_j}^{\tau_{j + 1}} \chi(\tau, x^{[j]}(\tau), u^{[j]}(\tau), v_j) \, \rd \tau \\
            & \qquad + \frac{\nu_\varepsilon \bigl( (\tau_{j + 1}, x_{\tau_{j + 1}}(\cdot)), (\tau_{j + 1}, x^{[j]}_{\tau_{j + 1}}(\cdot)) \bigr)}{\varepsilon} + \frac{\zeta (\tau_{j + 1} - \tau_j)}{4 T} \\
            & \quad = \rho_\varepsilon^-(\tau_j, x_{\tau_j}(\cdot))
            - \frac{\nu_\varepsilon \bigl( (\tau_j, x_{\tau_j}(\cdot)), (\tau_j, x^{[j]}_{\tau_j}(\cdot)) \bigr)}{\varepsilon}
            - \int_{\tau_j}^{\tau_{j + 1}} \chi(\tau, x^{[j]}(\tau), u^{[j]}(\tau), v_j) \, \rd \tau \\
            & \qquad + \frac{\nu_\varepsilon \bigl( (\tau_{j + 1}, x_{\tau_{j + 1}}(\cdot)), (\tau_{j + 1}, x^{[j]}_{\tau_{j + 1}}(\cdot)) \bigr)}{\varepsilon} + \frac{\zeta (\tau_{j + 1} - \tau_j)}{4 T}.
        \end{align*}
        Thus, we find that, if we introduce the function
        \begin{displaymath}
            \begin{aligned}
                \omega(\tau)
                & \doteq \frac{\nu_\varepsilon \bigl( (\tau, x_\tau(\cdot)), (\tau, x^{[j]}_\tau(\cdot)) \bigr)}{\varepsilon} \\
                & \quad + \int_{\tau_j}^{\tau} \bigl( \chi(\xi, x(\xi), u(\xi), v(\xi))
                - \chi(\xi, x^{[j]}(\xi), u^{[j]}(\xi), v_j) \bigr) \, \rd \xi
                \quad \forall \tau \in [\tau_j, \tau_{j + 1}],
            \end{aligned}
        \end{displaymath}
        then, in order to prove \cref{inequality_step_main}, it remains to verify that
        \begin{equation} \label{kappa_main}
            \omega(\tau_{j + 1}) - \omega(\tau_j)
            \leq \frac{\zeta(\tau_{j + 1} - \tau_j)}{4 T}.
        \end{equation}

        In accordance with \cref{lemma_nu_advanced}, (b), the function $\omega(\cdot)$ is Lipschitz continuous and (see \cref{motion_arbitrary_ell,h})
        \begin{displaymath}
            \begin{aligned}
                \dot{\omega}(\tau)
                & = h \biggl( \tau, x(\tau), u(\tau), v(\tau), \frac{\boldsymbol{\nabla} \mu_\varepsilon^{(\tau, x^{[j]}_\tau(\cdot))} (\tau, x_\tau(\cdot))}{\varepsilon} \biggr) \\
                & \quad - h \biggl( \tau, x^{[j]}(\tau), u^{[j]}(\tau), v_j, \frac{\boldsymbol{\nabla} \mu_\varepsilon^{(\tau, x^{[j]}_\tau(\cdot))} (\tau, x_\tau(\cdot))}{\varepsilon} \biggr)
                \quad \text{for a.e. } \tau \in [\tau_j, \tau_{j + 1}].
            \end{aligned}
        \end{displaymath}
        Hence, due the choice of $\delta$, we have
        \begin{equation} \label{dot_kappa_1}
            \begin{aligned}
                \dot{\omega}(\tau)
                & \leq h \biggl( \tau_j, x(\tau_j), u(\tau), v(\tau), \frac{\boldsymbol{\nabla} \mu_\varepsilon^{(\tau_j, r^{[j]}(\cdot))} (\tau_j, x_{\tau_j}(\cdot))}{\varepsilon} \biggr) \\
                & \quad - h \biggl( \tau_j, r^{[j]}(\tau_j), u^{[j]}(\tau), v_j, \frac{\boldsymbol{\nabla} \mu_\varepsilon^{(\tau_j, r^{[j]}(\cdot))} (\tau_j, x_{\tau_j}(\cdot))}{\varepsilon} \biggr)
                + \frac{\zeta}{8 T}
            \end{aligned}
        \end{equation}
        for a.e. $\tau \in [\tau_j, \tau_{j + 1}]$.
        Since $u(\tau) = U_\varepsilon^0(\tau_j, x_{\tau_j}(\cdot))$ for all $\tau \in [\tau_j, \tau_{j + 1})$ by \cref{control_law_first_player}, using definitions \cref{U_circ} and \cref{Hamiltonian} of $U_\varepsilon^0$ and $H$ (see also \cref{upper_Hamiltonian}), we derive
        \begin{equation} \label{dot_kappa_2}
            \begin{aligned}
                & h \biggl( \tau_j, x(\tau_j), u(\tau), v(\tau), \frac{\boldsymbol{\nabla} \mu_\varepsilon^{(\tau_j, r^{[j]}(\cdot))} (\tau_j, x_{\tau_j}(\cdot))}{\varepsilon} \biggr) \\
                & \quad \leq H\biggl( \tau_j, x(\tau_j),
                \frac{\boldsymbol{\nabla} \mu_\varepsilon^{(\tau_j, r^{[j]}(\cdot))} (\tau_j, x_{\tau_j}(\cdot))}{\varepsilon} \biggr).
            \end{aligned}
        \end{equation}
        On the other hand, in a similar way, the choice of $v_j$ implies that (see also \cref{lower_Hamiltonian})
        \begin{equation} \label{dot_kappa_3}
            \begin{aligned}
                & h \biggl( \tau_j, r^{[j]}(\tau_j), u^{[j]}(\tau), v_j, \frac{\boldsymbol{\nabla} \mu_\varepsilon^{(\tau_j, r^{[j]}(\cdot))} (\tau_j, x_{\tau_j}(\cdot))}{\varepsilon} \biggr) \\
                & \quad \geq H\biggl( \tau_j, r^{[j]}(\tau_j),
                \frac{\boldsymbol{\nabla} \mu_\varepsilon^{(\tau_j, r^{[j]}(\cdot))} (\tau_j, x_{\tau_j}(\cdot))}{\varepsilon} \biggr).
            \end{aligned}
        \end{equation}
        From \cref{dot_kappa_1,dot_kappa_2,dot_kappa_3,H-H}, it follows that $\dot{\omega}(\tau) \leq \zeta / (4 T)$ for a.e. $\tau \in [\tau_j, \tau_{j + 1}]$, which yields \cref{kappa_main} and, thus, completes the proof.
    \end{proof}

    Thus, according to \cref{theorem_positional}, the players can ensure the achievement of the game value $\rho^0$ with any given accuracy $\zeta > 0$ by using the corresponding feedback control laws $(U^0_\varepsilon, \Delta)$ and $(V^0_\varepsilon, \Delta)$.
    In particular, at every time $\tau_j$, $j \in \overline{1, p - 1}$, of the partition $\Delta$, the players choose their controls $u(\tau)$ and $v(\tau)$ for $\tau \in [\tau_j, \tau_{j + 1})$ based on the information about the motion history $x_{\tau_j}(\cdot)$ on $[0, \tau_j]$ only.

    In this regard, it is appropriate to make one more comment on \cref{assumption_K_2}.
    Suppose that the free term $y(\cdot)$ and the kernel $K$ from the integral equation \cref{system} meet some additional smoothness conditions that provide the correctness of further reasoning.
    First of all, note that, if we differentiate equation \cref{system}, we get
    \begin{equation*}
        \dot{x}(\tau)
        = \dot{y}(\tau)
        + K(\tau, \tau) f(\tau, x(\tau), u(\tau), v(\tau))
        + \int_{0}^{\tau} \frac{\partial K(\tau, \xi)}{\partial \tau} f(\xi, x(\xi), u(\xi), v(\xi)) \, \rd \xi
    \end{equation*}
    for a.e. $\tau \in [0, T]$.
    Taking this into account, let us introduce an additional variable $z(\cdot)$ as a solution of the ordinary differential equation
    \begin{equation} \label{reduction_z}
        \dot{z}(\tau)
        = f(\tau, x(\tau), u(\tau), v(\tau))
        \quad \text{for a.e. } \tau \in [0, T]
    \end{equation}
    under the initial condition $z(0) = 0$.
    Then, using the integration by parts formula, we obtain
    \begin{equation} \label{reduction_x}
        \begin{aligned}
            \dot{x}(\tau)
            & = \dot{y}(\tau)
            + K(\tau, \tau) f(\tau, x(\tau), u(\tau), v(\tau)) \\
            & \quad + \frac{\partial K(\tau, \tau)}{\partial \tau} z(\tau)
            - \int_{0}^{\tau} \frac{\partial^2 K(\tau, \xi)}{\partial \xi \, \partial \tau} z(\xi) \, \rd \xi
            \quad \text{for a.e. } \tau \in [0, T].
        \end{aligned}
    \end{equation}
    As a result, we see that the integral equation \cref{system} reduces to the system of two differential equations \cref{reduction_x,reduction_z} under the initial conditions $z(0) = 0$ and $x(0) = y(0)$.
    Since the last term from the right-hand side of equation \cref{reduction_x} depends on the values $z(\xi)$ for all $\xi \in [0, \tau]$, the system \cref{reduction_x,reduction_z} can be classified as a time-delay system.
    So, in order to study the original game \cref{system,cost_functional}, the results of the differential game theory developed for such systems can be applied.
    In particular, we can derive the existence of $\zeta$-optimal players' feedback strategies that, at each time $\tau \in [0, T)$, depend on the values $z(\xi)$, $\xi \in [0, \tau]$, and $x(\tau)$.
    Nevertheless, it should be emphasized that the additional variable $z(\cdot)$ is introduced artificially, and it does not seem reasonable to assume that the players are able to know the required values $z(\xi)$, $\xi \in [0, \tau]$.
    At this point, \cref{assumption_K_2} comes to the fore.
    Namely, it allows us to extract the values $f(\xi, x(\xi), u(\xi), v(\xi))$ for a.e. $\xi \in [0, \tau]$ and, as a consequence, the values $z(\xi)$ for all $\xi \in [0, \tau]$ from the history $x_\tau(\cdot)$ of the motion $x(\cdot)$ of the original system \cref{system}, information about which is available to the players.
    Finally, note that, under the assumptions made in the paper, the above reduction of the integral equation \cref{system} to the system  \cref{reduction_x,reduction_z} is not applicable, which is mainly due to the fact that the kernel $K$ may have a singularity.

\section{Conclusion}
\label{section_conclusion}

    In this paper, we have considered the zero-sum game for the weakly-singular Volterra integral equation of Hammerstein type \cref{system} and the cost functional \cref{cost_functional}.
    Under the additional assumption on the kernel of this equation (see \cref{assumption_K_2}), we have proved that this game has the value and proposed the method for constructing players' optimal positional strategies.
    To this end, we have developed the approach from the differential games theory that is based on the dynamic programming principle and involves the study of the corresponding Hamilton--Jacobi equations and their viscosity solutions.
    In particular, this required introducing the new class of path-dependent Hamilton--Jacobi equations with the special coinvariant derivatives and establishing some results on viscosity solutions of such equations.

    The paper seems useful for further investigations of zero-sum games for Volterra integral equations.
    In particular, it is of interest to obtain another characterizations of the value functional that are more convenient for verifying compared to the proposed definition of a viscosity solution (for example, in terms of suitable directional derivatives).
    This would open up opportunities for finding the value functional and constructing players' optimal positional strategies in specific examples.

\end{document}